\documentclass[a4paper,10pt]{amsart}

\usepackage[utf8]{inputenc}

\usepackage{graphicx}

\usepackage{cite}

\usepackage[english]{babel}

\usepackage{hyperref}

\hypersetup{%
pdftitle={},
pdfsubject={Mathematics},
pdfauthor={Manfred Madritsch},
pdfkeywords={}
hyperindex=true,plainpages=false}

\usepackage{a4wide}

\usepackage{amsmath}
\usepackage{amsfonts}
\usepackage{amssymb}
\usepackage{amsthm}
\usepackage[initials]{amsrefs}

\newtheorem{lem}{Lemma}[section]
\newtheorem{thm}[lem]{Theorem}
\newtheorem{prop}[lem]{Proposition}
\newtheorem{cor}[lem]{Corollary}

\numberwithin{equation}{section}

\newtheorem*{cor*}{Corollary}
\newtheorem*{thm*}{Theorem}

\theoremstyle{definition}
\newtheorem{defi}{Definition}[section]

\theoremstyle{remark}
\newtheorem{rem}[lem]{Remark}

\allowdisplaybreaks[3]

\newcommand{\N}{\mathbb{N}}
\newcommand{\Z}{\mathbb{Z}}
\newcommand{\Q}{\mathbb{Q}}
\newcommand{\R}{\mathbb{R}}

\newcommand{\lf}{\left\lfloor}
\newcommand{\rf}{\right\rfloor}
\renewcommand{\lvert}{\left\vert}
\renewcommand{\rvert}{\right\vert}
\renewcommand{\lVert}{\left\Vert}
\renewcommand{\rVert}{\right\Vert}

\title{Multidimensional van der Corput sets and small fractional parts
  of polynomials}

\author[M. G. Madritsch]{Manfred G. Madritsch}
\address[M. G. Madritsch]{
  \noindent 1. Universit\'e de Lorraine, Institut Elie Cartan de
  Lorraine, UMR 7502, Vandoeuvre-l\`es-Nancy, F-54506, France;\newline
  \noindent 2. CNRS, Institut Elie Cartan de Lorraine, UMR 7502,
  Vandoeuvre-l\`es-Nancy, F-54506, France}
\email{manfred.madritsch@univ-lorraine.fr}

\author[R. F. Tichy]{Robert F. Tichy} \address[R. F. Tichy]{Department
  for Analysis and Number Theory\\Graz University of
  Technology\\A-8010 Graz, Austria} \email{tichy@tugraz.at}

\subjclass[2010]{}

\keywords{Diophantine inequalities, van der Corput set, Heilbronn set,
  generalized prime powers, exponential sums}

\date{\today}

\begin{document}

\begin{abstract}
  We establish Diophantine inequalities for the fractional parts of
  generalized polynomials $f$, in particular for sequences
  $\nu(n)=\lfloor n^c\rfloor+n^k$ with $c>1$ a non-integral real
  number and $k\in\N$, as well as for $\nu(p)$ where $p$ runs through
  all prime numbers. This is related to classical work of Heilbronn
  and to recent results of Bergelson \textit{et al.}
\end{abstract}

\maketitle

\section{Introduction}\label{sec:introduction}

By Dirichlet's approxmiation theorem for any real number $\xi$ and any
positive integer $N$, 
\[\min_{1\leq n\leq N}\lVert n\xi\rVert\leq\frac1{N+1},\]
where $\lVert\cdot\rVert$ denotes the distance to the nearest
integer. Hardy and Littlewood
\cite{hardy_littlewood1914:some_problems_diophantine1} conjectured a
similar result for the distances $\lVert n^k\xi\rVert$ (with given
exponent $k\in\N$).
Vinogradov \cite{vinogradov1927:analytischer_beweis_des} proved the
following
\begin{thm}
  Let $\xi\in\R$ be a given real number and $N\in\N$, a given positive
  integer. Then for every $k\in\N$, there exists an exponent
  $\eta_k>0$ such that
  \[\min_{1\leq n\leq N}\lVert\xi n^k\rVert\ll_k N^{-\eta_k}.\]
\end{thm}
Throughout this paper we use the standard notation $\ll$, where the
index denotes the dependence of the implicit constant; furthermore we
use instead of $\ll$ sometimes the $\mathcal{O}$-notation.

In the case of squares Heilbronn
\cite{heilbronn1948:distribution_sequence_n} improved Vinogradov's
exponent to $-\frac12+\varepsilon$ (with arbitrary
$\varepsilon>0$). For the related literature up to 1986 we refer to
the classical monograph of Baker
\cite{baker1986:diophantine_inequalities}. The best known exponent in
the quadratic case is $-\frac47+\varepsilon$ due to Zaharescu
\cite{zaharescu1995:small_values_n}. However, his method is not
applicable to higher powers. It is an open conjecture that $\eta_k$
can be taken as $1-\varepsilon$ with arbitrary $\varepsilon>0$.

Generalizations to arbitrary polynomials $f\in\Z[X]$ with $f(0)=0$ are
due to Davenport~\cite{davenport1967:theorem_heilbronn} and
Cook~\cite{cook1972:fractional_parts_set}. Let
$2\leq k_1\leq\cdots\leq k_s$ be integers, $\xi_i\in\R$ for
$1\leq i\leq s$ and $\ell=2^{1-k_1}+\cdots+2^{1-k_s}$. Then Cook
\cite{cook1976:diophantine_inequalities_with} also showed that
$\lVert
\xi_1n_1^{k_1}+\cdots+\xi_sn_s^{k_s}\rVert\ll_{k_1,\ldots,k_s,\varepsilon}N^{-s/(s+(1-\ell)2^{k_s-1})+\varepsilon}$.
Wooley~\cite{wooley1995:new_estimates_smooth} considered the
Diophantine inequality over smooth numbers to obtain an improvement.
The proofs of these results are based on a sophisticated treatment of
the occurring exponential sums.  In a recent paper L{\^e} and
Spencer~\cite{le_spencer2014:intersective_polynomials_and} proved the
following
\begin{thm}[{\cite[Theorem
    3]{le_spencer2014:intersective_polynomials_and}}] \label{le:spencer:thm3}
  Let $N\in\N$ and $h\in\Z[X]$ be a polynomial with
  integer coefficients such that for every non-zero integer $q$ there
  exists a solution $n$ to the congruence $h(n)\equiv 0\bmod q$. Then
  there is an exponent $\eta>0$ depending only on the degree of
  $h$ such that
  \[\min_{1\leq n\leq N}\lVert\xi h(n)\rVert\ll_h N^{-\eta}\]
  for arbitrary $\xi\in\R$.
\end{thm}

A recent proof of a related well-known conjecture concerning
the Vinogradov integral is also due to
Wooley~\cite{wooley2016:cubic_case_main} for $k=3$ and to Bourgain
\textit{et al.}
\cite{bourgain_demeter_guth2016:proof_main_conjecture} in the general
case. Baker \cite{baker2016:small_fractional_parts} used this approach
to improve on Diophantine inequalities as considered in Theorem
\ref{le:spencer:thm3}.

Danicic \cite{danicic1959:fractional_parts_heta} considered two dimensional
extensions of the above problem and showed that
\[\min_{1\leq n\leq N}\max\left(\lVert \alpha n^2\rVert, \lVert \beta
  n^2\rVert\right)\ll N^{-\eta}\]
uniformly in $N$, $\alpha$ and $\beta$. Higher powers were
investigated by Liu \cite{liu1970:fractional_parts_heta}. Cook
\cite{cook1972:fractional_parts_set, cook1973:fractional_parts_set}
generalized these results to a system of polynomials without constant
term.

\begin{defi}\label{def:jointly_intersective}
  Let $h_1,\ldots,h_k$ be a system of polynomials in $\Z[X]$. This
  system is called jointly intersective if for every $q\neq0$, there
  exists an $n\in \Z$ such that $h_i(n)\equiv 0\pmod q$ for
  $i=1,\ldots k$.
\end{defi}

\begin{rem}
  The concept of jointly intersective polynomials was introduced
  independently (under different names) in different areas by
  Bergelson \textit{et
    al.}~\cite{bergelson_leibman_lesigne2008:intersective_polynomials_and},
  L{\^e}~\cite{le2014:problems_and_results},
  Rice~\cite{rice2013:sarkoezys_theorem_scr} and
  Wierdl~\cite{wierdl1989:almost_everywhere_convergence}.

  Note that the common root condition in Definition
  \ref{def:jointly_intersective} is necessary which is shown by a
  simple counter example in the case $\xi=a/q$.
\end{rem}

Now we link the concept of jointly intersective polynomials with
intersective sets. In the case $k=1$ an arbitrary subset
$\mathcal{I}\subset\Z\setminus\{0\}$ is called intersective if
\[\mathcal{I}\cap(S-S)\neq\emptyset\] whenever the upper density of $S\subset\Z$
is positive. A famous result, independently established by Furstenberg
\cite{furstenberg1977:ergodic_behavior_diagonal} and S\'ark\"ozy
\cite{sarkozy1978:difference_sets_sequences1} states that the set of
$k$-th powers and the set of shifted primes $p+1$ and $p-1$ are
intersective sets. Moreover, a sufficient condition for a set being
intersective is due to Kamae and
Mendès-France~\cite{kamae_mendes1978:van_der_corputs} and was later
extended by Nair~\cite{nair1998:uniformly_distributed_sequences,
  nair1998:uniformly_distributed_sequences2,
  nair1992:certain_solutions_diophantine}. L{\^e} and
Spencer~\cite{le_spencer2014:intersective_polynomials_and} established
the following
\begin{thm}[{\cite[Theorem
    4]{le_spencer2014:intersective_polynomials_and}}] \label{le:spencer:thm4}
  Let $\ell$ be a positive integer, $h_1,\ldots,h_k$ be jointly intersective
  polynomials, and let $A=(a_{ij})$ be an arbitrary $\ell\times k$
  matrix with real entries and let $N\in\N$.  Then there is an
  exponent $\eta>0$ depending only on $\ell$ and on the polynomials $h_i$
  such that
\[\min_{1\leq n\leq N}\max_{1\leq i\leq \ell}\lVert
\sum_{j=1}^ka_{ij}h_j(n)\rVert \ll_{\ell,h_1,\ldots,h_k}N^{-\eta},\]
where the bound is uniform in $A$.
\end{thm}

Harman \cite{harman1981:trigonometric_sums_over} considered the
sequence $\alpha p^k$ for $\alpha>0$ and $k$ a positive integer, where
$p$ runs through the prime numbers. Baker and Kolesnik
\cite{baker_kolesnik1985:distribution_p_alpha} considered the
distribution modulo one of the more general sequence
$\alpha p^\theta$. Improvements of the latter for the case $\theta=1$
have been established by Matom\"aki
\cite{matomaeki2009:distribution_alpha_p}. Recent refinements of the
statement are given by Baker
\cite{baker2017:fractional_parts_of}. Motivated by the above
observations concerning intersective sets, L{\^e} and Spencer
\cite{le_spencer2015:intersective_polynomials_and} also proved an
extension of Theorem \ref{le:spencer:thm4} to \textit{polynomials
  evaluated at prime numbers}.

The aim of our paper is to establish such Diophantine inequalities for
the Piatetski-Shapiro sequence and for pseudo-polynomial sequences,
for instance for the sequence $n^c+n^k$ with $c>0$ a non-integral real
number and $k\in\N$. In Section \ref{sec:van-der-corput} we introduce
the concept of van der Corput sets and we formulate our results in
detail. In Section \ref{sec:every-van-der} we prove that also in the
multi-variate setting every van der Corput set is a Heilbronn
set. This extends a classical one-dimensional result of Montgomery
\cite{montgomery1994:ten_lectures_on}. Section
\ref{sec:expon-sum-estim} is devoted to exponential sum estimates, the
final Sections \ref{sec:case-single-pseudo}
and~\ref{sec:multi-dimens-case} deal with single and multiple
pseudo-polynomials, respectively.

\section{Van der Corput sets and statement of Results}\label{sec:van-der-corput}

In the following we introduce van der Corput sets for multi-parameter
systems that is for $\Z^k$-actions (in the terminology of ergodic
theory). For more details see Bergelson and Lesigne
\cite{bergelson_lesigne2008:van_der_corput}.

\begin{defi}
A subset $\mathcal{H}\subset\Z^k\setminus\{\mathbf{0}\}$ is a van der Corput set (vdC-set) if
for any family $(u_{\mathbf{n}})_{\mathbf{n}\in\Z^k}$ of complex numbers of modulus 1 such
that
\[\forall \mathbf{h}\in\mathcal{H},\quad
\lim_{N_1,\ldots,N_k\to\infty}
\frac1{N_1\cdots N_k}
\sum_{0\leq \mathbf{n}<(N_1,\ldots,N_k)}u_{\mathbf{n}+\mathbf{h}}\overline{u_{\mathbf{n}}}=0\]
we have
\[\lim_{N_1,\ldots,N_k\to\infty}\frac1{N_1\cdots N_k}
\sum_{0\leq \mathbf{n}<(N_1,\ldots,N_k)}u_{\mathbf{n}}=0.\]
Here in the limit $N_1,\ldots,N_k$ tend to infinity independently and
$<$ stands for the product order.
\end{defi}
Equivalently, $\mathcal{H}$ is a vdC-set if any family
$(x_{\mathbf{n}})_{\mathbf{n}\in\N^k}$ of real numbers having the
property that for all $\mathbf{h}\in\mathcal{H}$ the family
$(x_{\mathbf{n}+\mathbf{h}}-x_{\mathbf{n}})_{\mathbf{n}\in\N^k}$ is
uniformly distributed mod 1, is itself uniformly distributed mod
1. The concept of uniform distribution for multi-parameter systems (so
called multi-sequences) was investigated in various papers, see for
instance Losert and Tichy
\cite{losert_tichy1986:uniform_distribution_subsequences},
Kirschenhofer and
Tichy~\cite{kirschenhofer_tichy1981:uniform_distribution_double},
Tichy and Zeiner \cite{tichy_zeiner2010:baire_results_multisequences}
and the book of Drmota and Tichy
\cite{drmota_tichy1997:sequences_discrepancies_and}.

Van der Corput's difference theorem states that in the case $k=1$ the
full set $\N$ of positive integers is a van der Corput set
(\textit{cf}. {\cite[Theorem
  3.1]{kuipers_niederreiter1974:uniform_distribution_sequences}}).
However, this is only a sufficient condition. Therefore the question
of the necessary ``size'' of vdC sets arises. For various aspects of
van der Corput's difference theorem we refer to the recent paper of
Bergelson und Moreira
\cite{bergelson_moreira2016:van_der_corputs}. Delange observed that
also the sets $q\N$, where $q\geq2$ is an integer, are van der Corput
sets. More general examples like the $k$th powers or shifted primes
$p+1$ and $p-1$ are due to Kamae and
Mend{\`e}s-France~\cite{kamae_mendes1978:van_der_corputs}.  In
particular, they proved in the case $k=1$ that each van der Corput set
is also intersective. The converse does not hold true as it was shown
by Bourgain \cite{bourgain1987:ruzsas_problem_on}.

In his seminal paper Ruzsa
\cite{ruzsa1984:connections_between_uniform} gave four equivalent
definitions of vdC-sets. We refer the interested reader to chapter 2
of Montgomery \cite{montgomery1994:ten_lectures_on} or the important
work of Bergelson and
Lesigne~\cite{bergelson_lesigne2008:van_der_corput} for a detailed
account on vdC-sets.

\begin{defi}\label{def:heilbronn:set}
Let $\mathcal{H}\in\Z^k\setminus\{\mathbf{0}\}$. We call $\mathcal{H}$ a
Heilbronn set if for every $\xi\in\R^k$ and every
$\varepsilon>0$ there is an $\mathbf{h}\in\mathcal{H}$ such that
\[\lVert\mathbf{h}\cdot\xi\rVert<\varepsilon\]
where $\cdot$ denotes the standard inner product.
\end{defi}

Following Montgomery we want to analyze quantitative aspects of
Heilbronn sets. This is related to the spectral definition of van der
Corput sets given by Kamae and Mend\'es-France
\cite{kamae_mendes1978:van_der_corputs} and its multidimensional
variant by Bergelson and
Lesigne~\cite{bergelson_lesigne2008:van_der_corput}. For each subset
$\mathcal{H}\subset\Z^k\setminus\{\mathbf{0}\}$ we denote by
$\mathcal{T}=\mathcal{T}(\mathcal{H})$ the set of real trigonometric
polynomials
\[T(\mathbf{x})=a_0+\sum_{\mathbf{h}\in \mathcal{H}}a_{\mathbf{h}}\cos(2\pi \mathbf{h}\cdot
\mathbf{x})\]
with $T(\mathbf{x})\geq0$ for all $\mathbf{x}\in\R^k$ and $T(\mathbf{0})=1$. Furthermore we
set
\[\delta(\mathcal{H}):=\inf_{T\in\mathcal{T}(\mathcal{H})} a_0.\]

In order to provide a quantitative result on Heilbronn sets we
introduce the following quantity.
\begin{gather}\label{eta}
  \gamma=\gamma(\mathcal{H})=\sup_{\xi\in\R^k}\inf_{\mathbf{h}\in
    \mathcal{H}} \lVert\mathbf{h} \cdot \xi\rVert.
\end{gather}

Then our first result is the following
\begin{thm}\label{thm:vdC_Heilbronn}
Let $\mathcal{H}\subset\Z^k\setminus\{0\}$.
\begin{enumerate}
\item $\mathcal{H}$ is a van der Corput set if and only if
  $\delta(\mathcal{H})=0$.
\item $\mathcal{H}$ is a Heilbronn set if and only if
  $\gamma(\mathcal{H})=0$.
\item $\gamma(\mathcal{H})\leq \delta(\mathcal{H})$.
\item Any van der Corput set is a Heilbronn set.
\end{enumerate}
\end{thm}

\begin{rem}
  In the one dimensional case the if-statement of Assertion (1) was
  already proved by Kamae and Mend\'es-France
  \cite{kamae_mendes1978:van_der_corputs} and the only-if-part is due
  to Ruzsa \cite{ruzsa1984:connections_between_uniform}. The general
  case was shown Bergelson and
  Lesigne~\cite{bergelson_lesigne2008:van_der_corput}. Assertion (2)
  is obvious. The one dimensional version of Assertion (3) was already
  shown by Montgomery \cite{montgomery1994:ten_lectures_on}. In
  Section \ref{sec:every-van-der} we prove the multidimensional case
  of Assertion (3). Finally Assertion (4) is a direct consequence of
  Assertion (3).
\end{rem}

In Montgomery \cite{montgomery1994:ten_lectures_on} one can also find
counter examples which show that in general the converse is
false. Former results by Bergelson, Boshernitzan, Kolesnik, Lesigne,
Madritsch, Quas, Son, Tichy and Wierdl provide us with many examples
of vdC-sets:
\begin{itemize}
\item Let $f,g\in\Z[X]$ such that for any $q\in\N$ there exists
  $n\in\N$ such that $q$ divides $f$ and $g$. Then the $2$-dimensional
  set $\{(f(n),g(n))\colon n\in\N\}$ is a vdC-set
  (see~\cite{bergelson_lesigne2008:van_der_corput}).
\item Let $f,g\in\Z[X]$ be polynomials with zero constant term. Then
  the $2$-dimensional set $\{(f(p-1),g(p-1))\colon p\text{ prime}\}$
  is a vdC-set (see~\cite{bergelson_lesigne2008:van_der_corput}).
\item Let $c>1$ be irrational and $b\neq0$. Then the set
  $\{\lfloor bn^c\rfloor\colon n\in\N\}$ is a vdC-set
  (see~\cite{boshernitzan_kolesnik_quas+2005:ergodic_averaging_sequences}).
\item Let $b,d\neq0$ such that $b/d$ is irrational, $c\geq1$, $a>0$
  and $a\neq c$. Then the set
  $\{\lfloor bn^c+dn^a\rfloor\colon n\in\N\}$ is a vdC-set
  (see~\cite{boshernitzan_kolesnik_quas+2005:ergodic_averaging_sequences}).
\item Let $b\neq 0$, $c>1$ be irrational and $d$ any real number. Then
  the set $\{\lfloor bn^c(\log n)^d\rfloor\colon n\in\N\}$ is a vdC-set
  (see~\cite{boshernitzan_kolesnik_quas+2005:ergodic_averaging_sequences}).
\item Let $b\neq 0$, $c>1$ be rational and $d\neq0$. Then the set
  $\{\lfloor bn^c(\log n)^d\rfloor\colon n\in\N\}$ is a vdC-set
  (see~\cite{boshernitzan_kolesnik_quas+2005:ergodic_averaging_sequences}).
\item Let $b,d\neq0$, $c\geq1$ and $a>1$. Then the set
  $\{\lfloor bn^c+d(\log n)^a\rfloor\colon n\in\N\}$ is a vdC-set
  (see~\cite{boshernitzan_kolesnik_quas+2005:ergodic_averaging_sequences}).
\item Let $\alpha_i$ be positive integers and $\beta_i$ be positive
  non-integral reals. Then the $(k+\ell)$-dimensional sets
  \[\left\{\left((p-1)^{\alpha_1},\ldots,(p-1)^{\alpha_k},\lfloor
    (p-1)^{\beta_1}\rfloor,\ldots,\lfloor
    (p-1)^{\beta_\ell}\rfloor\right)\colon p\text{ prime}\right\}\]
  and 
  \[\left\{\left((p+1)^{\alpha_1},\ldots,(p+1)^{\alpha_k},\lfloor
      (p+1)^{\beta_1}\rfloor,\ldots,\lfloor
      (p+1)^{\beta_\ell}\rfloor\right)\colon p\text{ prime}\right\}\]
  are vdC-sets (see
  \cite{bergelson_kolesnik_madritsch+2014:uniform_distribution_prime}
  and \cite{madritsch_tichy2016:dynamical_systems_and}).
\end{itemize}
Furthermore Bergelson \textit{et al.}
\cite{bergelson_kolesnik_son2015:uniform_distribution_subpolynomial}
showed under some mild conditions that for a function $f$ from a
Hardy-field the sequence $(\{f(p)\})_{p\text{ prime}}$ is uniformly
distributed mod 1. From this they deduced general classes of vdC-sets
and by our Theorem \ref{thm:vdC_Heilbronn} they are Heilbronn sets, too.

In Sections \ref{sec:case-single-pseudo} and
\ref{sec:multi-dimens-case} we want to show that for sets of the form
$\{\lfloor n^c\rfloor+n^k\colon n\in\N\}$, where $c>1$ is not an
integer and $k\in\N$, and multidimensional variants thereof, we may
replace $\varepsilon$ by some negative power $N^{-\eta}$ depending
only on the exponents $c$ and $k$. Therefore we want to introduce the
concept of pseudo-polynomials.

\begin{defi}
  Let
  $\alpha_1,\alpha_2,\ldots,\alpha_d,\theta_1,\theta_2,\ldots,\theta_d$
  be positive reals such that $1\leq\theta_1<\cdots<\theta_d$ and at
  least one $\theta_j\not\in\Z$ for $1\leq j\leq d$. A function
  $f\colon\R\to\R$ of the form
  \begin{gather*}
    f(x)=\alpha_1x^{\theta_1}+\cdots+\alpha_dx^{\theta_d}
  \end{gather*}
  is called a pseudo-polynomial.

  By abuse of notation we write $\deg f=\theta_d$.
\end{defi}
We investigate Diophantine inequalities for sequences of the form
$(\lf f(n)\rf)_{n\geq1}$. A simple example is the Piatetski-Shapiro type
sequence $(\lf n^c\rf+n^k)_{n\geq1}$ with $c>1$ a non-integral real
and $k\in\N$.

\begin{thm}\label{thm:single_heilbronn_set}
  Let $\xi$ be a real number, $N\in\N$ sufficiently large and $f$ be a
  pseudo-polynomial. Then there exists an exponent $\eta>0$ depending
  only on $f$ such that
  \[\min_{1\leq n\leq N}\lVert\xi\lf f(n)\rf\rVert\ll_f N^{-\eta}.\]
\end{thm}

When applying our result to the sequence $(\lf n^c\rf+n^k)_{n\geq1}$ we
obtain the following
\begin{cor}
  Let $\xi$ be real, $c>1$ be a non-integral real number, $N\in\N$
  sufficiently large and $k\in\N$. Then for arbitrary $\varepsilon>0$
\[\min_{1\leq n\leq N}\lVert \xi\lf n^c+n^k\rf\rVert\ll
\begin{cases}N^{-\frac12\frac1{2^{\lceil c\rceil+1}-1}+\varepsilon}
  &\text{if }c>k,\\
N^{-\frac{1}{2^{k-1}(k+2)}+\varepsilon}
&\text{if }c<k.
\end{cases}\]
\end{cor}

\begin{rem}
  Using the methods of Mauduit and
  Rivat~\cite{mauduit_rivat1995:repartition_des_fonctions,
    mauduit_rivat2005:proprietes_q_multiplicatives} and of
  Morgenbesser~\cite{morgenbesser2011:sum_digits_lfloor} combined with
  Spiegelhofer~\cite{spiegelhofer2014:piatetski_shapiro_sequences} and
  M{\"u}llner and
  Spiegelhofer~\cite{muellner_spiegelhofer2015:normality_thue}, who
  considered Beatty sequences, one could improve this exponent in the
  case $c>k$.
\end{rem}

Similarly to the results above we may consider sequences over the primes.

\begin{thm}\label{thm:single_prime_heilbronn_set}
  Let $\xi$ be a real number, $N\in\N$ sufficiently large and $f$ be a
  pseudo-polynomial. Then there exists an exponent $\eta>0$ depending
  only on $f$ such that
  \[\min_{\substack{1\leq p\leq N\\p\text{ prime}}}\lVert\xi\lf f(p)\rf\rVert\ll_f N^{-\eta}.\]
\end{thm}

\begin{cor}
  Let $\xi$ be real, $c>1$ be a non-integral real number, $N\in\N$ sufficiently large and
  $k\in\N$. Then for arbitrary $\varepsilon>0$
\[\min_{\substack{1\leq p\leq N\\p\text{ prime}}}\lVert \xi\lfloor p^c+p^k\rfloor\rVert\ll
\begin{cases}
N^{-\frac12\frac1{2^{\lceil c\rceil+1}-1}+\varepsilon} &\text{if
}c>k,\\
N^{-\frac{1}{4^{k-1}(k+2)}+\varepsilon}
&\text{if }c<k.
\end{cases}
\]
\end{cor}

Finally we state multidimensional variants of these estimates.

\begin{thm}\label{thm:multiple_heilbronn_set}
  Let $\ell$ be a positive integer, $f_1,\ldots,f_k$ be $\Q$-linearly
  independent pseudo-polynomials, let $A=(a_{ij})$ be an arbitrary
  $\ell\times k$ matrix with real entries and let $N\in\N$.  Then
  there is an exponent $\eta>0$ depending only on $\ell$ and on the
  polynomials $f_i$ such that
  \[\min_{1\leq n\leq N}\max_{1\leq i\leq \ell}\lVert
  \sum_{j=1}^ka_{ij}\lf f_j(n)\rf\rVert \ll_{\ell,f_1,\ldots,f_k}N^{-\eta},\]
  where the bound is uniform in $A$.
\end{thm}

Again we may consider sequences of integer parts of pseudo-polynomials
over the primes.

\begin{thm}\label{thm:multiple_prime_heilbronn_set}
  Let $\ell$ be a positive integer, $f_1,\ldots,f_k$ be $\Q$-linearly
  independent pseudo-polynomials, let $A=(a_{ij})$ be an arbitrary
  $\ell\times k$ matrix with real entries and let $N\in\N$.  Then
  there is an exponent $\eta>0$ depending only on $\ell$ and on the
  polynomials $f_i$ such that
  \[\min_{1\leq n\leq N}\max_{1\leq i\leq \ell}\lVert
  \sum_{j=1}^ka_{ij}\lf f_j(p)\rf \rVert \ll_{\ell,f_1,\ldots,f_k}N^{-\eta},\]
  where the bound is uniform in $A$.
\end{thm}

\section{Every van der Corput set is also a Heilbronn set}\label{sec:every-van-der}

We first present the following easy proof\footnote{Personal
  communication by Imre Rusza} that any van der Croput set is a Heilbronn
set. Let $\mathcal{H}\subset\Z^k\setminus\{0\}$ be a van der Corput set. It is an
easy deduction that $\mathcal{H}$ must be intersective in the
multidimensional sense. Then for a given $\xi\in\R^k$ we consider the
set
\[A=\{\mathbf{h}\in\Z^k\colon
  \lVert\mathbf{h}\cdot\xi\rVert<\varepsilon/2\}.\] Since
$\mathcal{H}$ is intersective and $A$ has positive density we deduce
that
\[\mathcal{H}\cap(A-A)\neq\emptyset,\]
which proves Theorem \ref{thm:vdC_Heilbronn}.

A quantitative analysis of van der Corput and Heilbronn sets can be
provided using the parameter $\delta$ introduced in Section
\ref{sec:van-der-corput}. Let $N$ be a positive real. Then for a fixed
van der Corput set $\mathcal{H}$ we set
$\mathcal{H}_N:=\mathcal{H}\cap[-N,N]^d$. Then $\delta(\mathcal{H}_N)$
tends to $0$ as $N$ tends to infinity. Only few upper bounds for this
quantity are known. In particular, for the set of shifted primes
Slijep\v cevi\' c \cite{slijepcevic2013:van_der_corput} could show that
\[\delta(\{p-1\leq N\colon p\text{ prime}\})\ll (\log N)^{-1+o(1)}.\] Similar results are known for squares
\cite{slijepcevic2010:van_der_corput}. In case of Heilbronn sets one
has to consider the parameter $\gamma$ introduced in Section
\ref{sec:van-der-corput} instead of $\delta$.

The following proposition immediately yields a complete proof of
Theorem \ref{thm:vdC_Heilbronn}.
\begin{prop}\label{prop:vdC:Heilbronn}
  Let $\mathcal{H}\in\Z^k\setminus\{\mathbf{0}\}$. Then
  $\gamma(\mathcal{H})\leq\delta(\mathcal{H})$.
\end{prop}

\begin{proof}
  Our proof follows Montgomery's proof for $k=1$ in
  \cite{montgomery1994:ten_lectures_on}. Let
  $T\in\mathcal{T}(\mathcal{H})$ as above, let
  $0<\varepsilon\leq\frac12$ and set
  $f(x)=\max\left(0,1-\lVert x\rVert/\varepsilon\right)$.  Then we
  consider
  \[g(\xi):=a_0+\sum_{\mathbf{h}\in\mathcal{H}}a_{\mathbf{h}}f(\mathbf{h}
    \cdot \xi),\] where $\xi\in\R^k$. Since $f$ is continuous and of
  bounded variation, its Fourier transform converges absolutely to
  $f$. Thus
  \[g(\xi)=a_0+\sum_{\mathbf{h}\in\mathcal{H}}a_{\mathbf{h}}\sum_{m\in\Z} \widehat{f}(m) e(m\mathbf{h} \cdot \xi).\]
  The function $f$ is even. Hence its Fourier coefficients
  $\widehat{f}(m)$ are real. Moreover $g(\xi)$ is real, hence
  \[g(\xi)=a_0+\sum_{\mathbf{h}\in\mathcal{H}}a_{\mathbf{h}}\sum_{m\in\Z} \widehat{f}(m) \cos(2\pi m\mathbf{h} \cdot \xi).\] Inverting the order of summation yields
  \[g(\xi)=\sum_{m\in\Z}\widehat{f}(m)T(m\xi).\]
  A simple calculation shows that
  $\widehat{f}(m)=\frac1{\varepsilon}\left(\frac{\sin(\pi
      m\varepsilon)}{\pi m}\right)\geq0$
  and $T(m\theta)\geq0$ for all $m$. Thus $g(\xi)$ is greater than the
  contribution of the term for $m=0$ in the above sum. Since
  $\widehat{f}(0)=\varepsilon$ and $T(0)=1$ we get
  \[g(\xi)\geq\varepsilon.\] Now, if $\varepsilon>a_0$, then there
  must be at least one $\mathbf{h}\in\mathcal{H}$ such that
  $a_{\mathbf{h}}>0$ and $f(\mathbf{h} \cdot \xi)>0$. Hence
  $\lVert\mathbf{h} \cdot \xi\rVert<\varepsilon$. Since this holds for
  every $\varepsilon>a_0$ we obtain
  $\inf_{\mathbf{h}\in\mathcal{H}}\lVert \mathbf{h} \cdot
  \xi\rVert\leq a_0$.  Furthermore, since this holds for every
  polynomial $T\in\mathcal{T}(\mathcal{H})$, we get
  $\inf_{\mathbf{h}\in\mathcal{H}}\lVert\mathbf{h} \cdot
  \xi\rVert\leq\delta$. Finally, since this holds for every
  $\xi\in\R^k$, it follows that $\gamma\leq\delta$ which proves the
  proposition.
\end{proof}

\begin{rem}
  Applying the results of Slijep\v cevi\' c
  \cite{slijepcevic2013:van_der_corput,slijepcevic2010:van_der_corput}
  Proposition \ref{prop:vdC:Heilbronn} implies upper bounds for
  $\gamma$ in case of shifted primes and squares.
\end{rem}

\section{Exponential sum estimates for the case $\theta_r>k$}\label{sec:expon-sum-estim}

Before stating the proofs of the main theorems we collect some
well-known facts on exponential sums which will occur in the
sequel. Let $f$ be a pseudo-polynomial. Then there exists a real
function $g$ and a polynomial $P$ such that $f(x)=g(x)+P(x)$,
$g(x)=\sum_{j=1}^r\alpha_jx^{\theta_j}$ with
$1<\theta_1<\ldots<\theta_r$ and $\theta_j\not\in\Z$ for
$j=1,\ldots,r$. Let $k$ be the degree of $P$, and we set $k=0$ if
$P\equiv0$. By abuse of notation we write
\[\deg f=\begin{cases}
    \theta_r &\text{if }\theta_r>k\\
    k &\text{otherwise}.\end{cases}\]

We only consider the one dimensional case, since the multidimensional
case is similar. The proof is by supposing that
$\lVert \xi \lfloor f(n) \rfloor\rVert>M^{-1}$
for every $1\leq n\leq N$. Then using Vinogradov's method we
approximate the indicator function and by our assumption this yields a
lower bound for an exponential sum of the form
\[\sum_{1\leq n\leq N}e\left(\beta f(n)\right),\]
where $\beta>0$ is some real number.

The aim of this section is to provide upper bounds for this
exponential sum that subsequentially violates the lower bounds
yielding a contradiction. In the proof we use different arguments but
essentially the same exponential sum estimate if $\beta$ is large or
$\beta$ is small respectively, where these sizes have to be understood
modulo $1$. In both cases we use Weyl differencing which is
equivalent to derivation. If $g$ is the dominant part of $f$, this
means that $\deg f\not\in\Z$, we may differentiate as often as we wish
till the resulting function has the desired behavior. On the other
hand if the polynomial part $P$ is dominat we cannot differentiation that
freely since after $k$ steps we lose the polynomial and therefore the
dominant part. Therefore we have to further consider two cases for 
$k>\theta_r$ or not.

Let $\rho$ be a real satifying
\begin{gather}\label{rho}
0<\rho<\frac1{\lf\deg f\rf+3}.
\end{gather}
Then we distinguish the
following cases:
\begin{itemize}
\item If $\deg f=\theta_r$ is not an integer ($\theta_r>k$), then we
  only have
  \[N^{\rho-\theta_r}<\lvert\beta\rvert\leq N^{1/10}.\]
\item If $\deg f=k$ is an integer ($k>\theta_r$), then in the
  following section we distinguish the cases
\[N^{\rho-k}<\lvert\xi\rvert\leq N^{\rho-\theta_r},\quad
N^{\rho-\theta_r}<\lvert\xi\rvert\leq N^{1/10}.\]
\end{itemize}
Note that in the case of
\[0<\lvert \beta\rvert<N^{\rho-\deg f}\] we apply a differnt argument
that allows us to reuse the estimates for bigger $\beta$. Furthermore
we note that the exponent $\frac1{10}$ is an artifact of Lemma 2.3 of
\cite{bergelson_kolesnik_madritsch+2014:uniform_distribution_prime}
which we use in the proof.

If $\theta_r>k$ we may apply Weyl-differencing sufficiently often till
the sum does not rotate to much.
\begin{lem}{\cite[Lemma
    2.5]{bergelson_kolesnik_madritsch+2014:uniform_distribution_prime}}
\label{bkmst:lem2.5}
Let $X,k,q\in \mathbb{N}$ with $k,q\geq 0$ and set $K=2^k$ and $Q=
2^q$. Let $P(x)$ be a polynomial of degree $k$ with real
coefficients. Let $g(x)$ be a real $(q+k+2)$ times continuously
differentiable function on $[X , 2X]$ such that $\left| g^{(r)}(x)
\right| \asymp G X^{-r}$ $( r = 1, \dots, q+k+2) $.  Then, if $G =
o(X^{q+2})$ for $G$ and $X$ large enough, we have
\[\sum_{X < n \leq 2X} e(g(n) + P(n))\ll X^{1 -
  \frac{1}{K}} + X \left( \frac{\log^k X}{G} \right)^{\frac{1}{K}} + X
\left( \frac{G}{X^{q+2}} \right)^{\frac{1}{(4KQ-2K)}},\]
where $A\asymp B$ means that $A$ is of the same order as $B$,
\textit{i.e.} $A\ll B$ and $B\ll A$.
\end{lem}

\begin{prop}\label{prop:largeintegers}
  Let $P(x)$ be a polynomial of degree $k$ and
  $g(x) = \sum_{j=1}^r d_j x^{\theta_j}$ with $r \geq 1$, $d_r \ne 0$,
  $d_j$ real, $0 < \theta_1 < \theta_2 < \cdots < \theta_r$ and
  $\theta_j \notin \N$. Assume that $\ell <\theta_r<\ell +1$ for some
  integer $\ell$. Let 
  $N^{\rho-\theta_r} \leq |\xi| \leq N^{\frac1{10}}$. Then for
  arbitrary $\varepsilon>0$
  \[\left| \sum_{n \leq N} e(\xi g(n) + \xi P(n))\right|
    \ll N^{1-\frac{1}{(8KL-4K)}},\]
  where $K=2^k$ and $L=2^\ell$.
\end{prop}

\begin{proof}
We split the sum into $\leq\log N$ sub-sums of the form
\[\sum_{X\leq n\leq 2X}e(\xi g(n)+\xi P(n)),\]
and denote by $S$ a typical one of them. Because of the polynomial
structure of $g$ and since $\theta_r\not\in\Z$ we
get for $j\geq1$ that
\[\lvert g^{(j)}(x)\rvert\asymp X^{\theta_r-j}.\] Thus an application
of Lemma \ref{bkmst:lem2.5} with $q=\ell$ yields
\[
  S\ll X^{1 - \frac{1}{K}}
    + X \left( \frac{\log^k X}{\lvert \xi\rvert X^{\theta_r}} \right)^{\frac{1}{K}}
    + X \left( \frac{\lvert \xi\rvert X^{\theta_r}}{X^{q+2}} \right)^{\frac{1}{(4KQ-2K)}}.
\]
Summing over all sub-sums and using the above bounds on $\lvert
\xi\rvert$ we get
\begin{align*}
\sum_{n\leq N}e(\xi g(n)+\xi P(n))
\ll N \left(\lvert \xi\rvert N^{\theta_r}\right)^{-\frac{1}{K}}
    + N^{1-\frac{1}{(8KL-4K)}}
\ll N^{1-\frac{1}{(8KL-4K)}},
\end{align*}
which yields the desired bound.
\end{proof}

Before we will prove the corresponding estimate for sums over primes we need
some standard tools. The first one is the von Mangoldt's function defined by
\[
\Lambda(n)=\begin{cases}
\log p,&\text{if $n=p^k$ for some prime $p$ and an integer $k\geq1$;}\\
0,&\text{otherwise}.
\end{cases}
\]
\begin{lem}[{\cite[Lemma 11]{mauduit_rivat2010:sur_un_probleme}}]
\label{mr:lem11}
Let $g$ be a function such that $\lvert g(n)\rvert\leq 1$ for all
integers $n$. Then
\[
\lvert\sum_{p\leq P}g(p)\rvert\ll\frac1{\log P}\max_{t\leq P}
\lvert\sum_{n\leq t}\Lambda(n)g(n)\rvert+\sqrt{P}.
\]
\end{lem}





Next we introduce the $s$-fold divisor sum $d_s(n)$, 
\textit{i.e.}
\[d_s(n)=\sum_{x_1\cdots x_s=n}1.\]

The central tool in the treatment of the exponential sum over primes
is Vaughan's identity which we use to subdivide the weighted
exponential sum into several sums of Type I and II. 
\begin{lem}[{\cite[Lemma
    2.3]{bergelson_kolesnik_madritsch+2014:uniform_distribution_prime}}]
\label{bkmst:lem23}
Assume $F(x)$ to be any function defined on the real line, supported
on $[X, 2X]$ and bounded by $F_0$. Let further $U,V,Z$ be any
parameters satisfying $3 \leq U < V < Z < 2X$, $Z \geq 4U^2$,
$X \geq 32 Z^2 U$, $V^3 \geq 32 X$ and $Z-\frac12\in\mathbb{N}$. Then
$$\left| \sum_{X< n\leq 2X} \Lambda(n) F(n)  \right| \ll K
\log P + F_0 +  L (\log X)^8 ,$$
where $K$ and $L$ are defined by
\begin{align*}
K&=\max_M\sum_{m=1}^\infty d_3(m)\left\vert\sum\limits_{Z<n\leq M}
   F(mn)\right\vert& &\textnormal{(Type I)},\\
L&=\sup\sum_{m=1}^\infty d_4(m)\left\vert\sum\limits_{U < n < V} b(n)
   F(mn)\right\vert& &\textnormal{(Type II)},
\end{align*}
where the supremum is taken over all arithmetic functions $b(n)$ satisfying $|b(n)| \leq d_3(n).$
\end{lem}


Using these tools we obtain a similar estimate for the sum over primes.
\begin{prop}\label{prop:largeprimes}
  Let $P(x)$ be a polynomial of degree $k$ and
  $g(x) = \sum_{j=1}^r d_j x^{\theta_j}$ with $r \geq 1$, $d_r \ne 0$,
  $d_j$ real, $0 < \theta_1 < \theta_2 < \cdots < \theta_r$ and
  $\theta_j \notin \N$. Assume that $\ell <\theta_r<\ell +1$ for some
  integer $\ell$. Let 
  $N^{\rho-\theta_r} \leq |\xi| \leq N^{\frac1{10}}$. Then for
  arbitrary $\varepsilon>0$
  \[\left| \sum_{p \leq N} e(\xi g(p) + \xi P(p))\right| \ll
    N^{1-\rho+\varepsilon}+N^{1-\frac{2}{3K}+\varepsilon}
    +N^{1-\frac{\rho}{K}+\varepsilon}
    +N^{1-\frac{1}{64KL^5-4K}+\varepsilon},\]
  where $K=2^k$ and $L=2^\ell$.
\end{prop}

\begin{proof}
  We start with an application of Lemma \ref{mr:lem11} which
  transforms the sum over the primes into the weighted sum
  \[\left| \sum_{p \leq N} e(\xi g(p) + \xi P(p))\right|\ll\frac1{\log
    N}\max\lvert\sum_{n\leq
    N}\Lambda(n)e\left(\xi(g(n)+P(n))\right)\rvert+N^{\frac12}.\]
  Then we split the inner sum into $\leq \log N$ subsums of the form
  \[\lvert \sum\limits_{X< n \leq
    2X}\Lambda(n)e\left(\xi(g(n)+P(n))\right)\rvert\]
  with $2X \leq N$ and we denote by $S$ a typical one of them. We may assume
  that $X \geq N^{1-\rho}$.

  Applying Vaughan's identity (Lemma~\ref{bkmst:lem23}) with the parameters
  $U = \frac{1}{4} X^{1/5}$, $V= 4 X^{1/3}$ and $Z$ the unique number
  in $\frac12+\N$, which is closest to $\frac{1}{4} X^{2/5}$,
  yields
\begin{gather}\label{mani:S}
S \ll 1+(\log X)S_1+(\log X)^8S_2,
\end{gather}
where
\begin{align*}
S_1&=\sum_{x < \frac{2X}{Z}} d_3(x) \sum_{y > Z, \frac{X}{x} < y < \frac{2X}{x}} e\left(\xi(g(xy)+P(xy))\right)\\
S_2&=\sum_{\frac XV<x\leq\frac{2X}U} d_4(x) \sum_{U < y < V, \frac{X}{x} < y \leq \frac{2X}{x}} b(y) e\left(\xi(g(xy)+P(xy))\right).
\end{align*}

We start estimating $S_1$. Since $d_3(x)\ll
x^{\varepsilon}$ we have
\begin{align*}
\lvert S_1\rvert\ll X^\varepsilon\sum_{x\leq\frac{2X}Z}
  \lvert\sum_{\substack{\frac Xx<y\frac{2X}x\\y>Z}}e\left(\xi(g(xy)+P(xy))\right)\rvert.
\end{align*}
We fix $x$ and write $Y=\frac Xx$ for short. Since
$\theta_r\not\in\Z$, we obtain for $j\geq1$ that
\[\lvert\frac{\partial^j g(xy)}{\partial y^j}\rvert
\asymp X^{\theta_r}Y^{-j}.\]

For $j\geq5(\ell+1)$ we get
\[\xi X^{\theta_r}Y^{-j}\leq N^{\frac1{10}} X^{\theta_r-\frac25j}\ll X^{-\frac12}.\] 

Thus an application of Proposition \ref{bkmst:lem2.5} yields the following
estimate:
\begin{equation}\label{mani:estim:S1}
\begin{split}
\lvert S_1\rvert &\ll X^{\varepsilon}\sum_{x \leq 2X/Z} Y \left[
  Y^{-\frac{1}{K}} + (\log Y)^k\left(\lvert\xi\rvert X^{\theta}\right)^{-\frac1K} + X^{-\frac{1}{2} \frac{1}{4K \cdot 8L^5 - 2K}} \right] \\
&\ll X^{1+\varepsilon}(\log X)\left[X^{-\frac2{5K}}+X^{-\frac{\rho}{K}}
  + X^{-\frac{1}{64KL^5-4K} } \right],
\end{split}\end{equation}
where we have used that $\frac kK<1$ and
$\rho(\theta_r+1)<\rho(\lf\deg f\rf+2)<1$ by (\ref{rho}).

Now we turn our attention to the second sum $S_2$. We split the range
$( \frac{X}{V} , \frac{2X}{U} ]$ into $\leq \log X$ subintervals of
the form $(X_1, 2X_1]$. Thus
\begin{align*}
\lvert S_2\rvert
&\leq (\log X)X^{\varepsilon}\sum_{X_1<x\leq
  2X_1}\lvert\sum_{\substack{U<y<V\\\frac
    Xx<y\leq\frac{2X}x}}b(y)e\left(\xi(g(xy)+P(xy))\right)\rvert.
\end{align*}

An application of Cauchy's inequality together with the estimate $\lvert b(y)\rvert\ll
X^\varepsilon$ yields
\begin{align*}
\lvert S_2\rvert^2
&\leq (\log X)^2X^{2\varepsilon}X_1\sum_{X_1<x\leq
  2X_1}\lvert\sum_{\substack{U<y<V\\\frac
    Xx<y\leq\frac{2X}x}}b(y)e\left(\xi(g(xy)+P(xy))\right)\rvert^2\\
&\ll (\log X)^2X^{4\varepsilon}X_1\\
&\quad\times\left(X_1\frac{X}{X_1}+\lvert \sum_{X_1<x\leq2X_1} \sum_{A < y_1 < y_2 \leq B}e\left(\xi (g(xy_1)-g(xy_2) + P(xy_1)-P(xy_2))\right)  \rvert\right),
\end{align*}
where $A = \max \{U, \frac{X}{x} \} $ and
$B = \min \{U, \frac{2X}{x} \}$. A change of the order of summation yields
\begin{multline*}
|S_2|^2 \ll (\log X)^2X^{4\varepsilon}X_1\\
\times\left(X+
  \sum_{A < y_1 < y_2 \leq B}\lvert \sum_{X_1<x\leq2X_1} e\left(\xi (g(xy_1)-g(xy_2) + P(xy_1)-P(xy_2))\right)  \rvert\right).
\end{multline*}

We fix $y_1$ and $y_2 \ne y_1$. Similarly as above we get
\[\lvert\frac{\partial^j\left(g(xy_1)-g(xy_2)+P(xy_1)-P(xy_2)\right)}{\partial x^j}\rvert\asymp\frac{\lvert
  y_1-y_2\rvert}{y_1}X^{\theta_r}X_1^{-j}.\] In this case we suppose
that $j\geq2\lfloor\theta_r\rfloor+3$ in order to obtain
\[\xi\frac{\lvert
  y_1-y_2\rvert}{y_1}X^{\theta_r}X_1^{-j}
\ll X^{\frac1{10}+\theta_r}\left(\frac{X}{V}\right)^{-j}
\ll X^{\frac1{10}+\theta_r-\frac23j}
\ll X^{-\frac12}.\]
Thus again an application of Lemma \ref{bkmst:lem2.5} yields
\begin{equation}\label{mani:estim:S2}
\begin{split}
\lvert S_2\rvert^2 &\ll (\log X)^2X^{4\varepsilon}X_1\left(X +
  \sum_{A<y_1<y_2\leq B} X_1\left( X_1^{-\frac{1}{K}} +
    \left(\lvert\xi\rvert X^{\theta_r}\right)^{-\frac1{K}} + X^{-\frac{1}{2} \frac{1}{4K \cdot 2L^2 - 2K}}\right)\right) \\
&\ll (\log X)^2X^{4\varepsilon}\left(X^{2-\frac2{3K}} + X^{2-\frac{\rho}{K}} + X^{2- \frac{1}{16KL^2 - 4K}}\right).
\end{split}
\end{equation}

Plugging the two estimates \eqref{mani:estim:S1} and
\eqref{mani:estim:S2} into \eqref{mani:S} together with a summation
over all intervals proves the proposition.
\end{proof}

\section{Exponential sum estimates for the case $\theta_r<k$}
As above we write $f(x)=g(x)+P(x)$, where $P$ is a polynomial of
degree $k$ with real coefficients and
$g(x)=\sum_{j=1}^r\alpha_jx^{\theta_j}$ with
$1<\theta_1<\ldots<\theta_r$ and $\theta_j\not\in\Z$ for
$j=1,\ldots,r$. Then we consider two cases according to the size of
$\lvert\beta\rvert$:
\[N^{1-\rho-k}<\lvert\beta\rvert\leq N^{\rho-\theta_r}
  \quad\text{and}\quad
N^{\rho-\theta_r}<\lvert\beta\rvert\leq N^{1/10},\]
where $\rho$ is as in (\ref{rho}).

The``large'' coefficient case may be treated as in Section
\ref{sec:expon-sum-estim}. For the case of smaller coefficients we
need a completely different approach. On the one hand the coefficients
are too big ($\lvert\xi\rvert>N^{1-\rho-k}$) to use the method of van
der Corput and on the other hand the polynomial part is the dominant
one ($k>\theta_r$) and an application of Weyl differencing as in the
case of the large coefficients would make the polynomial disappear.

The idea is to separate the polynomial $P$ and the real function $g$
by means of a partial summation. Then for the sum over the polynomial
we use standard estimates due to Weyl
\cite{weyl1916:ueber_die_gleichverteilung} and
Harman~\cite{harman1981:trigonometric_sums_over} for the integer and
prime case, respectively. Since these are standard methods we present
only the non-standard steps and refer the interested reader to Chapter
3 of Montgomery~\cite{montgomery1994:ten_lectures_on}, Chapter 4 of
Nathanson \cite{nathanson1996:additive_number_theory} or the monograph
of Graham and Kolesnik \cite{graham_kolesnik1991:van_der_corputs} for
a more complete account on Weyl-van der Corput's method.

\begin{prop}\label{prop:mediumintegers}
Let $f$ be a pseudo polynomial and $\varepsilon>0$. Suppose that
$\rho(k+3)<1$ and that
\[N^{3\rho-k}\leq\lvert \xi\rvert \leq N^{\rho-\theta_r}.\]
Then for sufficiently large $N$
\[
\sum_{n\leq N}e\left(\xi f(n)\right)\ll N^{1-\rho2^{1-k}+\varepsilon}.
\]
\end{prop}

Before proving this first result we need some tools and definitions.
We start with the forward difference operator $\Delta_d$, which is the
linear operator defined on functions $f$ by the formula
\[\Delta_d(f)(x)=f(x+d)-f(x).\]
For $\ell\geq2$, we define the iterated difference operator
$\Delta_{d_\ell,d_{\ell-1},\ldots,d_1}$ by
\[\Delta_{d_\ell,d_{\ell-1},\ldots,d_1}
=\Delta_{d_\ell}\circ\Delta_{d_{\ell-1},\ldots,d_1}
=\Delta_{d_\ell}\circ\Delta_{d_{\ell-1}}\circ\cdots\circ\Delta_{d_1}.\]

The following lemma describes the idea of ``Weyl differencing''.
\begin{lem}[{\cite[Lemma 4.13]{nathanson1996:additive_number_theory}}]
\label{nat:lem4.13}
Let $N_1$,$N_2$,$N$, and $\ell$ be integers such that $\ell\geq1$,
$N_1<N_2$, and $N_2-N_1\leq N$. Let $f(n)$ be a real-valued arithmetic
function. Then
\[\lvert\sum_{n=N_1+1}^{N_2}e(f(n))\rvert^{2^\ell}
\leq(2N)^{2^\ell-\ell-1}\sum_{\lvert d_1\rvert<N}\cdots\sum_{\lvert
  d_\ell\rvert<N}\sum_{n\in
  I(d_\ell,\ldots,d_1)}e\left(\Delta_{d_\ell,\ldots,d_1}(f)(n)\right).\]
where $I(d_\ell,\ldots,d_1)$ is an interval of consecutive integers
contained in $[N_1+1,N_2]$.
\end{lem}

In order to estimate the innermost sum of the previous lemma we use
the following observation.
\begin{lem}[{\cite[Lemma 4.7]{nathanson1996:additive_number_theory}}]
\label{nat:lem4.7}
For every real number $\alpha$ and all integers $N_1<N_2$,
\[\sum_{n=N_1+1}^{N_2}e(\alpha n)
\ll \min\left(N_2-N_1,\lVert\alpha\rVert^{-1}\right).\]
\end{lem}

The final tool is the following estimate for sums of minima.
\begin{lem}[{\cite[Lemma 4.9]{nathanson1996:additive_number_theory}}]
\label{nat:lem4.9}
Let $\alpha$ be a real number. If
\[\lvert\alpha-\frac ab\rvert\leq\frac1{b^2},\]
where $b\geq1$ and $(a,b)=1$, then for any negative real numbers $H$
and $N$ we have
\[\sum_{h=1}^H\min\left(N,\frac{1}{\lVert\alpha h\rVert}\right)
\ll\left(b+H+N+\frac{HN}{b}\right)\max\{1,\log b\}.\]
\end{lem}

The following lemma links the estimate of the exponential sum with the
leading coefficient of the polynomial part and non-polynomial
part.
\begin{lem}\label{lem:mediumintegers}
  Let $X,k\in\N$ with $k\geq0$ and set $K=2^{k-1}$ Let $P$ be a
  polynomial of degree $k$ with real coefficients and let $\alpha$ be
  the leading coefficient. Let $g(x)$ be a real $k$ times continuously
  differentiable function on $[X,2X]$ such that
  $\lvert g^{(r)}(x)\rvert\asymp GX^{-r}$ ($r=1,\ldots,k$). Then, for
  $G$ and $X$ large enough, we have
  \begin{gather*}
    \lvert\sum_{X<n\leq 2X}e\left(P(n)+ g(n)\right)\rvert^{2^{k-1}}\ll
    X^{2^{k-1}-1}+(1+G)X^{2^{k-1}-k+\varepsilon}
    \sum_{t=1}^{k!X^{k-1}} \min\left(X,\frac1{\lVert
        t\alpha\rVert}\right)
  \end{gather*}
  with arbitrary $\varepsilon>0$.
\end{lem}

\begin{proof}
Our first tool is Lemma~\ref{nat:lem4.13} with
$\ell=k-1$ to get
\begin{multline*}
  \lvert \sum_{X<n\leq 2X}e\left(P(n)+ g(n)\right)\rvert^{2^{k-1}}\\
  \ll X^{2^{k-1}-1}+X^{2^{k-1}-k}\sum_{1\leq\lvert d_1\rvert<X}\cdots
  \sum_{1\leq\lvert d_{k-1}\rvert<X}\sum_{n\in I(d_{k-1},\ldots,d_1)}
  e\left(\Delta_{d_{k-1},\ldots,d_1}(P+g)(n)\right).
\end{multline*}

Now we want to separate the polynomial and non-polynomial part. To
this end we set
\[a_n=e(\Delta_{d_{k-1},\ldots,d_1}P(n))
\quad\text{and}\quad
b_n=e(\Delta_{d_{k-1},\ldots,d_1}g(n)).\]
Then an application of partial summation yields
\[\sum_{X<n\leq 2X}a_nb_n
\leq
\lvert\sum_{X<n\leq 2X}a_n\rvert+
\sum_{X<h\leq 2X}\lvert b_h-b_{h+1}\rvert
\lvert\sum_{X<n\leq h}a_n\rvert.\]

For the non-polynomial part $b_h-b_{h+1}$ we note the following representation for
the forward difference operator (\textit{cf.} Lemma 2.7 of Graham and
Kolesnik \cite{graham_kolesnik1991:van_der_corputs})
\[\Delta_{d_{k-1},\ldots,d_1}(g)(n)=
\int_0^1\cdots\int_0^1
\frac{\partial^{k-1}}{\partial t_1\cdots\partial
  t_{k-1}}g(n+t_1d_1+\cdots+t_{k-1}d_{k-1})
\mathrm{d}t_1\cdots\mathrm{d}t_{k-1}.\]
Together with the mean value theorem we get
\[\lvert b_h-b_{h+1}\rvert
\asymp X^{k-1}G X^{-k}
= GX^{-1}.\]

Now we turn our attention to the polynomial part, which means to the sum
of $a_n$. Since
$\Delta_{d_{k-1},\ldots,d_1}(P)(n)=k!d_1\cdots d_{k-1}\alpha
n+m(d_1,\ldots,d_{k-1})$, where $m$ is a function not depending on
$n$, an application of Lemma~\ref{nat:lem4.7} yields
\begin{align*}
\lvert\sum_{X<n\leq h}e\left(\Delta_{d_{k-1},\ldots,d_1} P(n)\right)\rvert
&\ll
\min\left(h,\frac1{\lVert k!d_1\cdots d_{k-1}\xi\alpha\rVert}\right).
\end{align*}

Putting the two estimates together we get
\begin{multline*}
  \lvert \sum_{X<n\leq 2X}e\left(P(n)+ g(n)\right)\rvert^{2^{k-1}}\\
  \ll X^{2^{k-1}-1}+X^{2^{k-1}-k}\sum_{1\leq\lvert d_1\rvert<X}\cdots
  \sum_{1\leq\lvert  d_{k-1}\rvert<X}\left(1+G\right)
  \min\left(X,\frac1{\lVert k!d_1\cdots d_{k-1}\alpha\rVert}\right).
\end{multline*}
Noting that for every $t\leq k!X^{k-1}$ there are $X^\varepsilon$
solutions $d_1,\ldots,d_{k-1}$ to
\[ t=k!d_1\ldots d_{k-1},\]
we get that
\[
\lvert \sum_{X<n\leq 2X}e\left(P(n)+ g(n)\right)\rvert^{2^{k-1}}
\ll X^{2^{k-1}-1}+(1+G)X^{2^{k-1}-k+\varepsilon}
  \sum_{t=1}^{k!X^{k-1}}
  \min\left(X,\frac1{\lVert t\alpha\rVert}\right),
\]
which proves the lemma.
\end{proof}

We have seen that the estimates reduce to an approximation of the
leading coefficient of the polynomial part of $f$. The following
lemma shows that the leading coefficient of $\xi P(x)$ can always be
approximated well provided that $\xi$ is in the ``medium'' range.
\begin{lem}\label{lem:approximation}
  Let $N$, $\alpha$, $\rho$ and $\xi$ be positive reals and let
  $k\geq2$ be a positive integer. Suppose that $\rho<(k+3)^{-1}$ and that
  \begin{gather*}
    N^{3\rho- k}<\lvert\xi\rvert\leq N^{\rho-\theta_r}.
  \end{gather*}
  Then there exist coprime $a,b\in\Z$ such that
  \[\lvert b\xi\alpha-a\rvert\leq b^{-2}
  \quad\text{and}\quad
  N^{2\rho}\leq b\leq N^{k-2\rho}
  \]
  provided that $N$ is sufficiently large.
\end{lem}

\begin{proof}
  By Dirichlet's approximation theorem there exist coprime $a,b\in \Z$
  such that
  \begin{gather*}
    \lvert b\xi\alpha
    - a\rvert\leq N^{-k+2\rho}
  \quad\text{and}\quad 1\leq b\leq N^{k-2\rho}.
  \end{gather*}
  If $b\geq N^{2\rho}$, then there is nothing to show. Suppose the
  contrary. We distinguish different cases for the size of $b$.
  \begin{itemize}
  \item[\textbf{Case 1.}] $2\leq b<N^{2\rho}$. In this case we get
    \[N^{\rho-\theta_r}\gg\lvert\xi\alpha\rvert\geq\lvert\frac
      ab\rvert-\frac1{b^2}\geq\frac1{2b}\geq\frac12N^{-2\rho}.\] Since
    $3\rho<1<\theta_r$, this contradicts our lower bound for
    sufficiently large $N$.
  \item[\textbf{Case 2.}] $b=1$. This case requires a further
    distinction according to whether $a=0$ or not.
    \begin{itemize}
    \item[\textbf{Case 2.1.}]  $\lvert\xi\alpha\rvert\geq\frac12$. It
      follows
      that \[N^{\rho-\theta_r}\gg\lvert\xi\alpha\rvert\geq\frac12\]
      which is absurd for $N$ sufficiently large.
    \item[\textbf{Case 2.2.}]  $\lvert\xi\alpha\rvert<\frac12$. This
      implies that $a=0$ which yields
      \begin{gather*}
        N^{3\rho-k}\ll\lvert\xi\alpha\rvert\leq N^{-k+2\rho},
      \end{gather*}
      which again is absurd for sufficiently large $N$.\qedhere
    \end{itemize}
  \end{itemize}
\end{proof}

Now we have all tools at hand to prove the estimate in the integer
case.
\begin{proof}[Proof of Proposition \ref{prop:mediumintegers}]
We divide the sum into $\leq(\log N)$ subsums of the form
\[\sum_{X<n\leq 2X}e(\xi f(n))\]
and denote by $S$ a typical one of them. Without loss of generality
we may suppose that $X\geq N^{1-\rho}$. For $r=1,\ldots,k$ we have
\[\lvert \xi g^{(\ell)}(n)\rvert\asymp\lvert \xi\rvert
  X^{\theta_r-\ell}.\]
Thus an application of Lemma \ref{lem:mediumintegers} yields
\[S^K\ll X^{K-1}+\left(1+\lvert \xi\rvert
  X^{\theta_r}\right)X^{K-k+\varepsilon}
  \sum_{t=1}^{k!X^{k-1}}\min\left(X,\frac{1}{\lVert t\xi\alpha\rVert}\right).
\]

Let $a,b$ be positive integers such that
\[\lvert\xi\alpha-\frac ab\rvert\leq\frac1{b^2}.\]
Using Lemma \ref{nat:lem4.9} we get for the sum
\[ S^K\ll X^{K-1}+\left(1+\lvert \xi\rvert
  X^{\theta_r}\right)X^{K-k+\varepsilon}
  \left(b+\frac{X^k}b\right).\]
Taking the $K$th root and summing over all dyadic intervals $[X,2X]$
we get
\[\sum_{n\leq N}e\left(\xi f(n)\right)
  \ll N^{1-\frac1K+\varepsilon}+N^{\frac{\rho}{K}}
  N^{1-\frac{k}{K}+\varepsilon}\left(b+\frac{N^k}{b}\right)^{\frac1K},\]
where we used that $\lvert \xi\rvert\leq N^{\rho-\theta}$. Finally we
derive the bounds for $b$ of Lemma \ref{lem:approximation}
\[\sum_{n\leq N}e\left(\xi f(n)\right)
  \ll
  N^{1-\frac1K+\varepsilon}+N^{\frac{\rho}{K}}N^{1+\varepsilon}N^{-\frac{2\rho}{K}}
  \ll N^{1-\frac{\rho}{K}}+\varepsilon,\]
which is the desired result.
\end{proof}

Now we turn our attention to the prime case. Therefore we will reuse the
central idea of the good approximation but we have to adopt the Weyl
differencing part. In particular, we will obtain bilinear forms (sums
over products) instead of the usual exponential sums.
\begin{prop}\label{prop:mediumprimes}
  Let $N$ and $\rho$ be positive reals and $f$ be a
  pseudo-polynomial. If $\rho(k+3)<1$ and $\xi$ is such that
\begin{gather*}
  N^{3\rho-k}<\lvert \xi \rvert\leq N^{\rho-\theta_r}
\end{gather*}
holds, then
\begin{gather*}
  \sum_{p\leq N}e\left(\xi f(p)\right)\ll
  N^{1-\rho4^{1-k}+\varepsilon}
\end{gather*}
  with arbitrary $\varepsilon>0$.
\end{prop}

For the prime case we need two further tools dealing with the bilinear
forms appearing after an application of Vaughan's identity (Lemma
\ref{bkmst:lem23}). These are adoptions of the corresponding Lemmas 3
and 4 of Harman~\cite{harman1981:trigonometric_sums_over} to our
setting of pseudo-polynomial functions. Let $\varphi$ and $\psi$ be
two real functions. Then for $U$, $V$ and $X$ reals, we consider sums
of the form
\begin{gather*}
\sum_{u=1}^U\sum_{\substack{v=1\\ X< uv\leq
      2X}}^V\varphi(u)\psi(v)e(P(uv)+g(uv)),
\end{gather*}
where $P$ is a polynomial of degree $k$ and $g$ is a $k$ times
continuously differentiable real function. Furthermore we define
$\Psi$ by
\begin{multline*}\Psi(n,y_1,\ldots,y_s)\\
=\psi(n)\prod_{i=1}^s\psi(n+y_i)\prod_{1\leq i<j\leq s}\psi(n+y_i+y_j)
\cdots\prod_{i=1}^s\psi\left(n+\sum_{j\neq i}y_i\right)
\psi\left(n+\sum_{i=1}^sy_i\right).
\end{multline*}

The first Lemma of Harman~\cite{harman1981:trigonometric_sums_over} is
the following.
\begin{lem}\label{har:cor:lem3}
  Let $P$ be a polynomial of degree $k$ with real coefficients and let
  $\alpha$ be its leading coefficient. Let $g(x)$ be a real ($2k+1$) times
  continuously differentiable function on $[X,2X]$ such that
  $\lvert g^{(r)}(x)\rvert \asymp GX^{-r}$ ($r=1,\ldots,k$). Set
  \[T=\max\lvert \psi(v)\rvert,\quad\text{and}\quad
    F=\left(\frac1U\left(\sum_{u\leq
          U}\varphi(u)^2\right)\right)^{\frac12}.\]
  For positive integers $U$, $V$, $X$ write
  \begin{gather}\label{bilinear:sum}
    S=\sum_{u=1}^U\sum_{\substack{v=1\\ X< uv\leq
        2X}}^V\varphi(u)\psi(v)e(P(uv)+g(uv)).
  \end{gather}
  Suppose that there exist $a,b\in\Z$ such that
  \[\lvert \alpha-\frac ab\rvert\leq\frac1{b^2}.\]
  Then we have
  \[\left(\frac{S}{TF}\right)^{K^2}\ll
  (UV)^{K^2+\varepsilon}\left(V^{-K}+(1+G)\left(U^{-1}+b^{-1}+(UV)^{-k}b\right)\right),\]
  where $K=2^{k-1}$.
\end{lem}

\begin{proof}
  We may assume that $T=F=1$ and $\psi(v)\geq0$ for all $v$ as well as
  omit the restriction $X\leq uv\leq 2X$ for the moment. Then an
  application of the Cauchy-Schwarz inequality yields
  \begin{align*}
    S^2
    &\ll U\sum_{v_1=1}^V\sum_{v_2=1}^V\psi(v_1)\psi(v_2)
      \sum_{u=1}^Ue\left(P(uv_1)-P(uv_2)+g(uv_1)-g(uv_2)\right)\\
    &\ll US_1+E_1.
  \end{align*}
  For positive integers $s$ we write for short
  \[S_s=\sum_{d_1=1}^{V-1}\cdots\sum_{d_s=1}^{V-1}
    \sum_{v}\Psi(v,d_1,\ldots,d_s)\sum_{u=1}^U
    e\left(\Delta_{d_s,\ldots,d_1}(P+g)(uv)\right),\]
  where the forward difference operator $\Delta_{d_s,\ldots,d_1}$ acts
  on $v$ not on $u$ and the range of summation over $v$ being
  $X<v<v+d_1+\cdots+d_s\leq 2X$, and
  \[E_s=U^{2^s}V^{2^s-1}.\]

  An easy induction shows for $s=2,\ldots,k-1$ that
  \[S^{2^s}\ll E_s+U^{2^s-1}V^{2^s-s-1}\lvert S_s\rvert.\]

  Now we look at the innermost sum of $S_s$. Since (\textit{cf.} Lemma
  10B of Schmidt \cite{schmidt1977:small_fractional_parts})
  \begin{align*}\Delta_{d_{k-1},\ldots,d_1}(P)(uv)
  &=d_1\cdots d_{k-}\left(\tfrac12k!\alpha
    u^k(2v+d_1+\cdots+d_{k-1})+(k-1)!\beta u^{k-1}\right)\\
  &=u^kh(d_1,\ldots,d_{k-1},v)+u^{k-1}(k-1)!\beta d_1\cdots d_{k-1},
  \end{align*}
  we can see $\Delta_{d_{k-1},\ldots,d_1}(P)(uv)$ as a polynomial of
  degree $k$ with leading coefficient
  $h(d_1,\ldots,d_{k-1},v)$. Furthermore
  $\Delta_{d_{k-1},\ldots,d_1}(g)(uv)$ is a $k$ times differentiable
  function and we may apply Lemma \ref{lem:mediumintegers}. Thus after
  $k-1$ iterations of the Cauchy-Schwarz inquality we obtain
  \begin{align*}
    S^{K^2}
    &\ll
    \left(UV\right)^{K^2}V^{-K}+U^{K^2-K}V^{K^2-k}
    \sum_{d_1=1}^V\cdots\sum_{d_{k-1}=1}^V\sum_{v}
    \lvert \sum_{u=1}^{U} 
      e\left(\Delta_{d_{k-1},\ldots,d_1}(P+g)(uv)\right)\rvert^{K}\\
    &\ll
    \left(UV\right)^{K^2}V^{-K}+(UV)^{K^2-k+\varepsilon}(1+G)
    \sum_{d_1,\ldots,d_{k-1}}\sum_{v}
    \sum_{t=1}^{k!U^{k-1}}\min\left(U,\frac1{\lVert th(d_1,\ldots,d_{k-1},v)\rVert}\right).
  \end{align*}

  There are at most $(UV)^\varepsilon$ ways of writing an number
  $t\leq (k!)^2V^kU^{k-1}$ as a product of the form
  \[\tfrac12 d_1d_2\cdots d_{k-1}(k!)(2v+d_1+\cdots+d_{k-1})=t.\]
  Thus
  \[
  S^{K^2}\ll
  \left(UV\right)^{K^2}V^{-K}+(UV)^{K^2-k+\varepsilon}(1+G)
  \sum_{m=1}^{(k!)^2V^kU^{k-1}}
     \min\left(U,\frac1{\lVert m\alpha\rVert}\right).
  \]

  Finally we use Lemma \ref{nat:lem4.9} for the sum of minima to get
  the desired bound.
\end{proof}

As second part we adapt Lemma 4 of Harman
\cite{harman1981:trigonometric_sums_over}.
\begin{lem}\label{har:lem4}
Suppose we have the hypotheses of Lemma \ref{har:cor:lem3}, but either
\begin{align*}
\varphi(x)&=1,\text{ for all $x$,}\\
\text{or }\varphi(x)&=\log x,\text{ for all $x$.}
\end{align*}
Then
\[
  S\ll (UV)^{1+\varepsilon}V^{\frac{k-1}{K}}(1+G)^{\frac1K}
  \left((UV)^{-k}b+U^{-1}+b^{-1}\right)^{\frac1K}.
\]
\end{lem}

\begin{proof}
  By an application of partial summation we may easily remove the
  $\log$ factor. Therefore without loss of generality we assume that
  $\varphi(x)=1$. Using H\"older's inequality we obtain
  \[
    S^{K}\leq V^{K-1}\sum_{v=1}^V \lvert\sum_{u=1}^Ue(P(uv)+g(uv))\rvert^{K},
  \]
  where this time the forward difference operator
  $\Delta_{d_1,\ldots,d_{k-1}}$ is with respect to $u$.

  Now an application of Lemma \ref{lem:mediumintegers} for the
  innermost sum yields
  \begin{align*}
    S^{K}
    &\ll V^{K-1}\sum_{v=1}^V
    U^{K-k+\varepsilon}(1+G)
    \sum_{t=1}^{k!U^{k-1}}\min\left(U,\frac1{\lVert tv^k\alpha \rVert}\right)\\
    &\ll U^{K-k+\varepsilon}V^{K-1}(1+G)\sum_{w=1}^{k!U^{k-1}V^k}
    \min\left(U,\frac1{\lVert w\alpha \rVert}\right).
  \end{align*}
  An application of Lemma \ref{nat:lem4.9} proves the lemma.
\end{proof}

Finally we have to combine the two lemmas as in the proof of Harman.
\begin{proof}[Proof of Proposition \ref{prop:mediumprimes}]
  This proof starts along the same lines as the proof of Proposition
  \ref{prop:largeprimes}. An application of Lemma \ref{mr:lem11} 
  transforms the sum over the primes into the weighted sum
  \[\left| \sum_{p \leq N} e(\xi g(p) + \xi P(p))\right|\ll\frac1{\log
    N}\max\lvert\sum_{n\leq
    N}\Lambda(n)e\left(\xi(g(n)+P(n))\right)\rvert+N^{\frac12}.\]
  Then we split the inner sum into $\leq \log N$ subsums of the form
  \[\lvert \sum\limits_{X< n \leq
    2X}\Lambda(n)e\left(\xi(g(n)+P(n))\right)\rvert\]
  with $2X \leq N$ and we denote by $S$ a typical one of them. We may assume
  that $X \geq N^{1-\rho}$.

  Applying Vaughan's identity (Lemma~\ref{bkmst:lem23}) with the parameters
  $U = \frac{1}{4} X^{1/5}$, $V= 4 X^{1/3}$ and $Z$ the unique number
  in $\frac12+\N$, which is closest to $\frac{1}{4} X^{2/5}$,
  yields
  \[
    S \ll 1+(\log X)S_1+(\log X)^8S_2,
  \]
  where
  \begin{align*}
    S_1&=\sum_{x < \frac{2X}{Z}} d_3(x) \sum_{y > Z, \frac{X}{x} < y < \frac{2X}{x}} e\left(\xi(g(xy)+P(xy))\right)\\
    S_2&=\sum_{\frac XV<x\leq\frac{2X}U} d_4(x) \sum_{U < y < V, \frac{X}{x} < y \leq \frac{2X}{x}} b(y) e\left(\xi(g(xy)+P(xy))\right).
  \end{align*}
  We consider these two sums as variants of the following general sum
  \[S_3=\sum_{u\leq \frac{2X}{V}}\sum_{\substack{v\leq V\\ X<uv\leq
        2X}}\varphi(u)\psi(v)e\left(\xi f(uv)\right),\]
  where $V\ll X^{\frac 13}$ or $V\ll X^{\frac 23}$.

  Similar as in Proposition \ref{prop:mediumintegers} we get that
  \[\lvert \xi g^{(\ell)}(n)\rvert\asymp \lvert \xi\rvert
    X^{\theta_r-\ell}.\] Furthermore let $a,b\in\Z$ be as in Lemma
  \ref{lem:approximation}, \textit{i.e.}
  \[\lvert \xi\alpha-\frac{a}{b}\rvert\leq b^{-2}
  \quad\text{and}\quad
  N^{2\rho}\leq b\leq N^{k-2\rho}.\]
  
  Now we apply Lemma \ref{har:cor:lem3} and Lemma \ref{har:lem4} whether 
  \begin{gather}\label{lem3orlem4}
    V^K\geq N^{-\rho}\min\left(b,N^{\frac13},N^kb^{-1}\right)
  \end{gather}
  holds or not respectively. Suppose that \eqref{lem3orlem4} holds. In
  this case an application of Lemma \ref{har:cor:lem3} yields
  \begin{align*}
    S_3^{K^2}
    \ll X^{K^2+\varepsilon}\left(V^K+\lvert\xi\rvert X^{\theta_r}\left(X^{-\frac13}+b^{-1}+X^{-k}b\right)\right)
  \end{align*}

  On the contrary, if (\ref{lem3orlem4}) does not hold, then an
  application of Lemm \ref{har:lem4} yields
  \[
    S_3\ll X^{1+\varepsilon} V^{\frac{k-1}K}(1+\lvert\xi\rvert X^{\theta_r})^{\frac1K}
      \left(X^{-k}b+X^{-\frac13}+b^{-1}\right)^{\frac1K}
  \]
  where we have used that $1/K-(k-1)/K^2\geq 1/K^2$. 

  Summing over the intervals $[X,2X]$ using both estimates together
  with $\lvert \xi\rvert\leq N^{\rho-\theta_r}$ yields the desired
  bound.
\end{proof}

\section{The case of a single pseudo-polynomial}\label{sec:case-single-pseudo}
In the present section we want to prove Theorems
\ref{thm:single_heilbronn_set} and
\ref{thm:single_prime_heilbronn_set}. The main tool originates from
large sieve estimates due to Montgomery which provides us with a lower
bound if all elements are sufficiently far away from an integer.
\begin{lem}\label{baker:thm2.2}
  Let $M$ and $N$ be positive integers. Consider a sequence of real
  numbers $x_1,\ldots,x_N$ and weights $c_1,\ldots,c_N\geq0$. Suppose
  $\lVert x_j\rVert\geq M^{-1}$ for all $j=1,\ldots,N$. Then there
  exists $1\leq m\leq M$ such that
\[\lvert \sum_{n=1}^N c_n e(mx_n)\rvert\geq
\frac1{6M} \sum_{n=1}^N c_n.\]
\end{lem}

\begin{proof}
  This is a weighted version of \cite[Theorem
    2.2]{baker1986:diophantine_inequalities}.
\end{proof}

The second tool are Vaaler polynomials which we need to deal with the
floor function.

  \begin{lem}[{\cite[Theorem
      19]{vaaler1985:some_extremal_functions}}]\label{vaaler:indicator:function}
    Let $I\subset[0,1]$ be an interval and $\chi_I$ its indicator
    function. Then for each positive integer $H$ there exist
    coefficients $a_H(h)$ and $C_h$ for $-H\leq h\leq H$ with
    $\lvert a_H(h)\rvert\leq 1$ and $\lvert C_h\rvert\leq 1$ such that
    the trigonometric polynomial
    \[\chi^*_{I,H}(t)=\lvert I \rvert+\frac1\pi\sum_{0<\lvert h\rvert\leq
      H}\frac{a_H(h)}{\lvert h\rvert} e(ht)\] satisfies
    \[\lvert \chi_I(t) - \chi^*_{I,H}(t)\rvert \leq
    \frac1{H+1}\sum_{\lvert h\rvert\leq H}C_h\left(1-\frac{\lvert
        h\rvert}{H+1}\right)e(ht).\]
  \end{lem}

\begin{rem}
  The coefficients $a_H(h)$ and $C_h$ are explicitly given in Vaaler's
  proof. However, the stated bounds are sufficient for our
  purposes.
\end{rem}

Now we are able to state the proof of Theorem
\ref{thm:single_heilbronn_set}.

\begin{proof}[Proof of Theorem \ref{thm:single_heilbronn_set}]
  Let $M=\lf N^\eta\rf$ where $\eta$ is a sufficiently small
  exponent. We conduct the proof by supposing that 
  \begin{gather}\label{farfrominteger}
    \min_{1\leq n\leq N}\lVert \xi\lf f(n)\rf\rVert\geq M^{-1}
  \end{gather}
  and deducing a contradiction. Since
  \[M^{-1}\leq\lVert \xi\lf f(n)\rf\rVert\leq\lVert \xi\rVert\cdot\lf
  f(n)\rf\ll\lVert \xi\rVert\] we get $\lvert\xi\rvert\gg
  M^{-1}$. Furthermore, by Lemma \ref{baker:thm2.2}, there exists
  $1\leq m\leq M$ such
  that \begin{gather}\label{mani:lower_bound_single}\lvert
    \sum_{n=1}^N e\left(m\xi\lf f(n)\rf\right)\rvert \gg \frac NM.
  \end{gather}
  The aim is to establish an upper bound of this exponential sum
  contradicting the lower bound for sufficiently small $\eta$.

We suppose for the moment that \[\lVert m\xi\rVert\geq N^{1-\rho-\deg f}.\]
The idea is to use digital expansion and Vaaler polynomials to get rid
of the floor function. Then we are left with an exponential sum and use
the tools from above.

Let $q\geq2$ be an integer, which is chosen later. Then we denote by $I_{d}$ with
$0\leq d< q-1$ the interval of all reals in $[0,1)$ whose initial
$q$-adic digit is $d$, \textit{i.e.}
  \[I_d:=\left[\frac dq,\frac{d+1}q\right).\]
  If $\{f(n)\}\in I_d$, then there exists $0\leq \vartheta<1$ such
  that $\{f(n)\}=\frac dq+\frac \vartheta q$. Thus
  \[e(m\xi\lf f(n)\rf) = e\left(m\xi f(n)-m\xi\frac
    dq\right)\left(1+\mathcal{O}\left(\frac1q\right)\right),\]
  yielding
  \begin{gather*}
    \lvert\sum_{n\leq N}e\left(m\xi\lf f(n)\rf\right)\rvert
    =\sum_{d=0}^{q-1}\lvert\sum_{n\leq N} e\left(m\xi
      f(n)\right)\chi_{I_d}(f(n))\rvert+\mathcal{O}\left(\frac
      Nq\right).
  \end{gather*}

Hence for a fixed $d\in\{0,\ldots,q-1\}$ we may write
\[
\lvert\sum_{n\leq N}
  e\left(m\xi f(n)\right)\chi_{I_d}(f(n))\rvert
\leq\lvert\sum_{n\leq N}
  e\left(m\xi f(n)\right)\chi^*_{I_d}(f(n))
\rvert+\sum_{n\leq N}\lvert \chi_{I_d}(f(n))-
\chi^*_{I_d,H}(f(n))\rvert,
\]
where we used the notation from Lemma
\ref{vaaler:indicator:function}. Using the estimates there we get for
the first part that
\begin{equation}\label{mani:expsum2}
\begin{split}
\lvert\sum_{n\leq N}
  e\left(m\xi f(n)\right)\chi^*_{I_d,H}(f(n))\rvert
&=\lvert\sum_{n\leq N}
  e\left(m\xi
    f(n)\right)\left(\frac1q+\frac1\pi\sum_{1\leq
      \lvert h\rvert\leq H}\frac{a_H(h)}{\lvert h\rvert}
    e(h f(n))\right)\rvert\\
&\leq\frac1q\lvert\sum_{n\leq N}
  e\left(m\xi f(n)\right)\rvert+
  \frac{1}{\pi} \sum_{0<\lvert h\rvert\leq
  H}\frac1h\lvert\sum_{n\leq N}
    e((m\xi+h) f(n))\rvert.
  \end{split}
\end{equation}

For the second part we again use the estimates in Lemma
\ref{vaaler:indicator:function} and arrive at
\begin{gather}\label{mani:expsum3}
\sum_{n\leq N}
\lvert \chi_{I_d}(f(n))-\chi^*_{I_d,H}(f(n)) \rvert
\leq\frac1{H+1}\sum_{\lvert h\rvert\leq H}\left(1-\frac{\lvert
      h\rvert}{H+1}\right)\lvert\sum_{n\leq N}e(hf(n))\rvert.
\end{gather}

The different exponential sums in (\ref{mani:expsum2}) and (\ref{mani:expsum3}) are of the form
\begin{gather}\label{mani:expsumforms}
  \sum_{n\leq N} e\left(m\xi f(n)\right),\quad
\sum_{n\leq N} e((m\xi+h) f(n))\quad\text{and}\quad
\sum_{n\leq N}e(hf(n)),
\end{gather}
respectively. We write them in the form $\sum_{n\leq N}e\left(\beta f(n)\right)$ for
short.

Since $\lVert m\xi\rVert\geq N^{1-\rho-\deg f}$ we get
$\lvert\beta\rvert\geq N^{1-\rho-\deg f}$ and we may distinguish the
following two cases.
\begin{itemize}
\item Suppose we have
  $N^{\rho-\theta_r}\leq\lvert\beta\rvert\leq N^{\frac1{10}}$. Then we
  write $f(x)=P(x)+g(x)$ where $P$ is a polynomial of degree $k$ and
  $g(x)=\sum_{j=1}^rd_jx^{\theta_j}$ with $1<\theta_1<\cdots<\theta_r$
  and $\theta_j\not\in\N$ for
    $1\leq j\leq r$. Now an application of Proposition~\ref{prop:largeintegers} yields
  \[\sum_{n\leq N}e\left(\beta f(n)\right)
  \ll N^{1-\frac{1}{8KL-4K}},\]
  where $K=2^k$ and $L=2^{\lf\theta_r\rf}$.
\item Now we suppose that
  $N^{1-\rho-\deg f}\leq\lvert\beta\rvert\leq N^{\rho-\theta_r}$. Then
  an application of Proposition~\ref{prop:mediumintegers} yields
  \[\sum_{n\leq N} e\left(\beta f(n)\right)\ll
  N^{1-\rho2^{1-k}+\varepsilon}\]
  with $\varepsilon>0$.
\end{itemize}
We set $\gamma=\min\left(\rho2^{1-k},(8KL-4K)^{-1}\right)$. Thus
\begin{align*}
\lvert\sum_{n\leq N}
  e\left(m\xi f(n)\right)\chi_{I_d}(f(n))\rvert
&\ll N^{1-\gamma}\left(\frac1q +\sum_{0<\lvert h\rvert\leq
  H}\frac1h+\frac1{H+1}\sum_{\lvert h\rvert\leq H}\left(1-\frac{\lvert h\rvert}{H+1}\right)\right)\\
&\ll N^{1-2\gamma+\varepsilon},
\end{align*}
where we have chosen $q=H=N^{\gamma}$. Finally we obtain
\[\lvert\sum_{n\leq N}e\left(m\xi\lfloor f(n)\rfloor\right)\rvert
  \ll N^{1-\gamma}.\] Plugging this upper
bound into the lower bound in (\ref{mani:lower_bound_single}) we get a
contradiction as soon as $\eta<\gamma$.

Now we turn our attention to the case of
$\lVert m\xi\rVert\leq N^{1-\rho-\deg f}$. Then there is some $h\in\Z$ such
that
\[\lvert\beta\rvert=\lvert m\xi+h\rvert\leq N^{1-\rho-\deg f}\]
and the coefficient in the second sum of (\ref{mani:expsumforms})
might be arbitraily small destroying our argument. Using the existence
of a multiple of $m$ in the sequence $(\lf f(n)\rf)_n$ yields a
contradiction to (\ref{farfrominteger}). In particular, for
sufficiently large $X$ (a small power of $N$ we will fix later) we
will show that there exists $1\leq n\leq X$ such that $\lf f(n)\rf$ is
a multiple of $m$.

Let
\[f(n)=a_\ell m^\ell+\cdots + a_1 m+a_0+a_{-1}m^{-1}+\cdots\]
be the $m$-adic expansion of $f(n)$. Then $\lf f(n)\rf$ is a multiple
of $m$ if and only if $a_0=0$ and this is the case if and only if
$m^{-1}f(n)\in[0,\frac1m)$.

Let $\mathcal{N}(f,X)$ be the number of $1\leq n\leq X$ such that
$\lf f(n)\rf$ is a multiple of $m$. Furthermore let $\chi$ be
the indicator function of the interval $[0,\frac1m)$. Then an
application of Vaaler polynomials (Lemma
\ref{vaaler:indicator:function}) yields
\begin{align*}\lvert\mathcal{N}(f,X)-\frac{X}m\rvert
  &\leq\sum_{1\leq n\leq
    X}\lvert\chi\left(\frac{f(n)}m\right)-\chi^*\left(\frac{f(n)}m\right)\rvert
  +\sum_{1\leq n\leq
    X}\lvert\chi^*\left(\frac{f(n)}m\right)-\frac1m\rvert\\
  &\leq \frac1{H+1}\sum_{\lvert h\rvert\leq H}C_h\left(1-\frac{\lvert
    h\rvert}{H+1}\right) \sum_{1\leq n\leq
    X}e\left(h\frac{f(n)}m\right)\\
  &\quad\quad+\frac1\pi\sum_{0<\lvert h\rvert\leq H}\frac{a_H(h)}{\lvert
    h\rvert}\sum_{1\leq n\leq X}e\left(h\frac{f(n)}m\right).
\end{align*}
Setting $X=M^{\frac1{1-\rho}}$ and $H=X^{\frac1{10}}$ we note that
\[X^{\rho-\theta_r}\leq X^{\rho-1}=\frac1M\leq\frac hm
  \leq X^{\frac1{10}}.\]
Then an application of Proposition \ref{prop:largeintegers} yields
\[\lvert\mathcal{N}(f,X)-\frac Xm\rvert\ll
  X^{1-\frac1{8KL-4K}+\varepsilon}.\]
Thus for sufficiently large $X$ (which is growing with $N$) we get
that $\mathcal{N}(f,X)\geq1$ and therefore there is a $1\leq n\leq X$
such that $\lf f(n) \rf=m\cdot r$ with $r\in \Z$. Thus
\[N^{-\eta}\leq\frac1M\leq \lVert \xi\lf f(n)\rf \rVert
\leq \lVert m\xi \rVert \cdot \lvert r \rvert
\leq N^{1-\rho-\deg f} M^{\frac{\deg f}{1-\rho}}\]
yielding a contradiction as long as
\[\eta<\frac{\deg f+\rho-1}{1+\frac{\deg f}{1-\rho}}.\]

Putting both cases together we get a contradction if 
\[\eta<\min\left(\rho2^{1-k},\frac1{8KL-4K}, 
  \frac{\deg f+\rho-1}{1+\frac{\deg f}{1-\rho}}\right),\]
proving the theorem.
\end{proof}

\begin{proof}[Proof of Theorem \ref{thm:single_prime_heilbronn_set}]
  This runs very much along the same lines as the proof of Theorem
  \ref{thm:single_heilbronn_set} above. Let $M=\lfloor N^\eta\rfloor$
  for a sufficiently small $\eta$ which we choose later. Suppose that
  $\lVert\xi\lfloor f(p)\rfloor\rVert\geq M^{-1}$ for all primes
  $2\leq p\leq N$. An application of Lemma \ref{baker:thm2.2} yields
  \begin{gather}\label{lower_bound_prime_single}
  \lvert\sum_{p\leq N}e\left(m\xi\lfloor f(p)\rfloor\right)\rvert\gg
  \frac{\pi(N)}{M},
  \end{gather}
  where $\pi$ is the prime-counting function. As in the integer case
  we are looking for an upper bound for the exponential sum yielding
  conditions on $\eta$.

  We start with the case of $\lVert m\xi\rVert\geq N^{1-\rho-\deg
    f}$. Following the lines of the integer case we have to find upper
  bounds for exponential sums of the form
  \[\sum_{p\leq N}e\left(\beta f(p)\right)\]
  with $\beta=m\xi$, $\beta=m\xi +h$ and $\beta=h$, respectively. We
  again distinguish two cases:
  \begin{itemize}
  \item Either we have $N^{1-\rho-\deg f}\leq \lvert\beta\rvert\leq
    N^{\rho-\theta_r}$. Then an application of Proposition
    \ref{prop:mediumprimes} yields
    \[\sum_{p\leq N}e\left(\beta f(p)\right)\ll N^{1-\rho4^{1-k}+\varepsilon}.\]
  \item Or we have $N^{\rho-\theta_r}\leq\lvert\beta\rvert\leq
    N^{\frac1{10}}$. Again we write $f(x)=P(x)+g(x)$ where $P$ is a
    polynomial of degree $k$ and   $g(x)=\sum_{j=1}^rd_jx^{\theta_j}$
    with $1<\theta_1<\cdots<\theta_r$ and $\theta_j\not\in\N$ for
    $1\leq j\leq r$. Then an application of Proposition
    \ref{prop:largeprimes} yields
    \[\sum_{p\leq N}e\left(\beta f(p)\right)\ll
      N^{1-\frac1{64KL^5-4K}+\varepsilon},\]
    where $K=2^k$ and $L=2^{\lf\theta_r\rf}$.
  \end{itemize}

  Now we turn our attention to the case of
  $\lVert m\xi \rVert\leq N^{1-\rho-\deg f}$. Again following the
  integer case above together with Proposition \ref{prop:largeprimes}
  we get the existence of a prime $2\leq p\leq M^{\frac1{1-\rho}}$
  such that $\lf f(p)\rf$ is a multiple of $m$. Repeating the steps
  from above we get a contradiction provided
  \[\eta<\frac{\deg f+\rho-1}{1+\frac{\deg f}{1-\rho}}.\]
  
  Together with the first case we get a contradiction provided that
  \[\eta<\min\left(\rho4^{1-k},\frac1{64KL^5-4K}, \frac{\deg
      f+\rho-1}{1+\frac{\deg f}{1-\rho}}\right)\] proving the theorem.
\end{proof}

\section{The multi-dimensional case}\label{sec:multi-dimens-case}
In this section we turn our attention to the case of simultaneous
approximation. The equivalent to
Theorem~\ref{thm:multiple_heilbronn_set} is the following more general
result.

\begin{thm}\label{thm:multiple_lattice_approximation}
  Let $f_1,\ldots,f_k$ be $\Q$-linear independent pseudo-polynomials,
  $\ell\in\N$ and $N\in\N$ sufficiently large. Then there exists
  $\theta_\ell>0$ such that for any lattice
  $\Lambda$ with determinant $\lvert\det(\Lambda)\rvert\leq N^{\theta_\ell}$, and
  any $\ell\times k$ matrix $A$ there exists $n\in\N$ with $1\leq
  n\leq N$ such that
  \[A\begin{pmatrix} \lfloor f_1(n)\rfloor \\ \vdots \\
    \lfloor f_k(n)\rfloor\end{pmatrix}\in\Lambda+B_\ell,\]
  where $B_\ell$ is the Euclidean unit ball in $\R^{\ell}$.
\end{thm}

The role of the lower bound similar to Lemma \ref{baker:thm2.2} is
played by the following multidimensional counterpart.

\begin{lem}[{\cite[Theorem
    14A]{schmidt1977:small_fractional_parts}}] \label{schmidt:thm14A}
  Suppose $\Lambda$ is a lattice of full rank in $\R^\ell$ such that
  $\Lambda\cap B_\ell=\{\mathbf{0}\}$, where $B_\ell$ denotes the
  Euclidean unit ball in $\R^\ell$. Suppose that
  $\mathbf{x}_1,\ldots,\mathbf{x}_N\in\R^\ell$ are not in
  $\Lambda+B_\ell$. Let $\varepsilon>0$ and
  \[S_{\mathbf{p}}=\sum_{n=1}^N e(\mathbf{x}_n\cdot\mathbf{p}).\]
  Then, provided $N$ is sufficiently large in terms of $\varepsilon$,
  there is a point $\mathbf{p}$ in a basis of the dual
  lattice $\Pi$ of $\Lambda$ such that
  $\lvert\mathbf{p}\rvert\leq N^\varepsilon$ and an integer
  $1\leq t\leq\frac{N^\varepsilon}{\lvert \mathbf{p}\rvert}$ such that
\[\lvert S_{t\mathbf{p}}\rvert\geq N^{1-\varepsilon}\lvert\det(\Lambda)\rvert^{-1}.\]
\end{lem}

\begin{proof}[Proof of Theorem \ref{thm:multiple_lattice_approximation}]
  As in the one-dimensional case we use Lemma \ref{schmidt:thm14A} to
  transform the problem into an estimation of an exponential sum. In
  particular, suppose that $\mathbf{x}_1,\ldots,\mathbf{x}_N$ are not
  in $\Lambda+B_\ell$. Then by Lemma \ref{schmidt:thm14A} there exists
  $\mathbf{p}$ with $\lvert\mathbf{p}\rvert\leq N^{\varepsilon}$ and
  an integer $1\leq m\leq\frac{N^\varepsilon}{\lvert\mathbf{p}\rvert}$
  such that
  \begin{gather}\sum_{n=1}^N e\left(mp_1\xi_1\lfloor
    f_1(n)\rfloor+\cdots+mp_k\xi_k\lfloor f_k\rfloor\right)\geq
  N^{1-\varepsilon}\det(\Lambda)^{-1}.
  \end{gather}
  Similar to above we derive an estimate for the exponential sum
  contradicting this lower bound.

  We start by getting rid of the floor function.
  %
  %
  Let $q\geq2$
  be an integer chosen later. Again we denote by $I_{d}$ the interval
  of all reals in $[0,1]$ whose $q$-adic expansion starts with $d$,
  \textit{i.e.}
  \[I_d:=\left[\frac dq,\frac{d+1}q\right).\]
  Then if $\{f_i(n)\}\in I_d$ there exists $0\leq \vartheta<1$ such that
  $\{f_i(n)\}=\frac dq+\frac \vartheta q$. Thus
  \[e(mp_i\xi_i\lf f_i(n)\rf)
  = e\left(mp_i\xi_if_i(n)-mp_i\xi_i\frac
    dq\right)\left(1+\mathcal{O}\left(\frac1q\right)\right).\] Hence
  \[\lvert\sum_{n\leq t}e\left(m\sum_{i=1}^kp_i\xi_i\lf f_i(n)\rf
      \right)\rvert
    =\sum_{d_1=0}^{q-1}\cdots\sum_{d_k=0}^{q-1}\lvert\sum_{n\leq t}
    e\left(m\sum_{i=1}^kp_i\xi_if_i(n)\right)\prod_{j=1}^k\chi_{I_{d_j}}(f_j(n))\rvert.
  \]

  As above we want to approximate the occurring indicator function by
  suitable functions. We follow Grabner
  \cite{grabner1989:harmonische_analyse_gleichverteilung}
  (\textit{cf.} Section 1.2.2 of Drmota and Tichy
  \cite{drmota_tichy1997:sequences_discrepancies_and}) who considered
  multidimensional variants of Vaaler polynomials. We fix a vector of
  digits $\mathbf{d}=(d_1,\ldots,d_k)$.  Thus
  \begin{multline}\label{applied:vaaler}
    \lvert\sum_{n\leq N}
    e\left(m\xi_1 f_1(n)+\cdots+m\xi_kf_k(n)\right)\prod_{j=1}^k\chi_{I_{d_j}}(f_j(n))\rvert\\
    \leq\lvert\sum_{n\leq N} e\left(m\xi_1
      f_1(n)+\cdots+m\xi_kf_k(n)\right)\prod_{j=1}^k\chi^*_{I_{d_j},H}(f_j(n))
    \rvert+R(H),
  \end{multline}
  where
  \begin{gather}\label{mani:R}
    R(H)=\sum_{n\leq N}\lvert\prod_{j=1}^k\chi_{I_{d_j}}(f_j(n))-
    \prod_{j=1}^k\chi^*_{I_{d_j},H}(f_j(n))\rvert.
  \end{gather}
  We start estimating $R(H)$. Using the inequality
  \[
  \lvert\prod_{j=1}^kb_j-\prod_{j=1}^ka_j\rvert
  \leq\sum_{\emptyset\neq J\subset\{1,\ldots,k\}}\prod_{j\not\in
    J}\lvert a_j\rvert \prod_{j\in J}\lvert b_j-a_j\rvert
  \]
  we get
  \begin{align*}
    \lvert\prod_{j=1}^k\chi_{I_j}(f_j(n))
    - \prod_{j=1}^k\chi^*_{I_j,H}(f_j(n))\rvert
    &\leq\sum_{\emptyset\neq J\subset\{1,\ldots,k\}}\prod_{j\in J}\lvert
      \chi_{I_j}-\chi^*_{I_j,H}\rvert\\
    &=\prod_{j=1}^k\left(1+\lvert
      \chi_{I_j}-\chi^*_{I_j,H}\rvert\right)-1.
  \end{align*}
  Plugging this into (\ref{mani:R}) together with Lemma
  \ref{vaaler:indicator:function} yields
  \begin{align*}
    &R(H)=\sum_{n\leq N}\lvert\prod_{j=1}^k\chi_{I_{d_j}}(f_j(n))-
      \prod_{j=1}^k\chi^*_{I_{d_j},H}(f_j(n))\rvert\\
    &\leq
      \sum_{n\leq N}\left(\prod_{j=1}^k\left(1+\lvert
      \chi_{I_j}-\chi^*_{I_j,H}\rvert\right)-1\right)\\
    &\leq
      \sum_{n\leq N}\left(\prod_{j=1}^k\left(1+\frac1{H+1}+\frac1{H+1}\sum_{1\leq\lvert
      h_j\rvert\leq H} C_{h_j}\left(1-\frac{\lvert
      h_j\rvert}{H+1}\right)e\left(h_jf_j(n)\right)\right)-1\right)\\
    &=N\left(\left(1+\frac1{H+1}\right)^k-1\right)
      +\sum_{0<\lVert\mathbf{h}\rVert_\infty\leq
      H}\left(\frac1{H+1}\right)^{k-\delta(\mathbf{h})}
      \left(1+\frac1{H+1}\right)^{\delta(\mathbf{h})}\lvert\sum_{n\leq N}e(\mathbf{h}\cdot\mathbf{f}(n))\rvert,
\end{align*}
where $\delta(\mathbf{h})=\sum_{j=1}^k\delta_{h_j0}$ counts the number of coordinates of
$\mathbf{h}=(h_1,\ldots,h_k)$ which are zero; $\lVert
\mathbf{h}\rVert_\infty=\max\{\lvert h_1\rvert,\ldots,\lvert h_k\rvert\}$.

Now we consider the first part of (\ref{applied:vaaler}). Again using
Lemma \ref{vaaler:indicator:function} we have
\begin{align*}
&\lvert\sum_{n\leq N}
  e\left(m\xi_1 f_1(n)+\cdots+m\xi_kf_k(n)\right)\prod_{j=1}^s\chi^*_{I_{d_j},H}(f_j(n))\rvert\\
&\quad=\lvert\sum_{n\leq N}
  e\left(m\xi_1
    f_1(n)+\cdots+m\xi_kf_k(n)\right)\prod_{j=1}^s\left(\frac1q+\frac1\pi\sum_{1\leq
      \lvert h_j\rvert\leq H}\frac{a_H(h_j)}{\lvert h_j\rvert}
    e(h_jf_j(n))\right)\rvert\\
&\quad\leq\frac1{q^s}\lvert\sum_{n\leq N}
  e\left(m\pmb{\xi}\cdot\mathbf{f}(n)\right)\rvert+
  \sum_{0<\lVert\mathbf{h}\rVert_\infty\leq H}\frac1{r(\pi\mathbf{h})}\lvert\sum_{n\leq N}
    e((m\pmb{\xi}+\mathbf{h})\cdot\mathbf{f}(n))\rvert,
\end{align*}
where $\pmb{\xi}=(\xi_1,\ldots,\xi_k)$ and
$\mathbf{f}(n)=(f_1(n),\ldots,f_k(n))$.

The occurring exponential sums are of the form
\[\sum_{n\leq N}e(\mathbf{h}\cdot\mathbf{f}(n)),\quad
  \sum_{n\leq N}
  e((m\pmb{\xi}+\mathbf{h})\cdot\mathbf{f}(n))\quad\text{or}\quad
  \sum_{n\leq N}
  e((m\pmb{\xi}+\mathbf{h})\cdot\mathbf{f}(n)).\]
We consider these three sums simultaneously and denote them by
\[\sum_{n\leq N}e\left(\sum_{i=1}^k\beta_if_i(n)\right).\]

By reordering if necessary we may suppose that
$\deg f_1<\deg f_2<\cdots<\deg f_k$. We split each
$f_i(n)=g_i(n)+P_i(n)$ where $P_i$ is a polynomial of degree $k_i$ and
$g_i(n)=\sum_{j=1}^r\alpha_jx^{\theta_{i,j}}$ with
$1<\theta_{i,1}<\cdots<\theta_{i,r}$ and $\theta_{i,j}\not\in\Z$ for
$j=1,\ldots,r$.

Let $\gamma_0=6\rho$ and suppose that
$\lVert\beta_1\rVert>N^{\gamma_0/2-\deg f_1}$. Then an application of
Proposition \ref{prop:largeintegers} and Proposition
\ref{lem:mediumintegers}, respectively, yields
\begin{align*}
  \sum_{n\leq N}e\left(\beta_1f_1(n)\right)\ll N^{1-\gamma_1},
\end{align*}
where
\[\gamma_1=\min\left(\rho2^{1-k_k},(8KL-4K)^{-1}\right).\]
Now we recursively define $\gamma_j$ for $2\leq j\leq k$ as
follows. Suppose that
$\lVert\beta_{j}\rVert>N^{\gamma_{j-1}/2-\deg f_{j}}$, then by
Proposition \ref{prop:largeintegers} and Proposition
\ref{lem:mediumintegers}, respectively, there exists $\gamma_j>0$ such
that
\[\sum_{n\leq N}e\left(\sum_{i=1}^{j}\beta_if_i(n)\right)
  \ll N^{1-\gamma_{j}}.\]

For the moment we suppose that there exists $1\leq s\leq k$ such that
$\lVert m\xi_j\rVert\leq N^{\gamma_{j-1}/2-\deg_j}$ for $s<j\leq k$ and
$\lVert m\xi_s\rVert>N^{\gamma_s-\deg f_s}$. Using partial summation
we split those $j>s$ apart. To this end we set
\begin{align*}
  \phi(n)=e\left(m\sum_{i=1}^sp_i\xi_i\lfloor f_i(n)\rfloor\right)
  \quad\text{and}\quad
  \psi(n)=e\left(m\sum_{i=s+1}^kp_{i}\xi_{i}f_{i}(n)\right).
\end{align*}
Since $\gamma_{j-1}>\gamma_{j}$ for $s< j\leq k$ we get
\[\lvert\psi(n+1)-\psi(n)\rvert\ll \sum_{i=s+1}^k N^{\gamma_{i-1}/2-\deg
    f_i}N^{\deg f_i-1}\ll N^{\gamma_{s}/2-1}.\] Thus
\begin{align*}
  \sum_{n\leq N}\phi(n)\psi(n)
  \leq \psi(N)\sum_{n\leq N}\phi(n)
  +N^{\gamma_{s}/2-1}\sum_{t\leq N}\sum_{n\leq t}\phi(n).
\end{align*}
By the definition of $\gamma_s$ we have
\[\sum_{n\leq N}\phi(n)=\sum_{n\leq
    N}e\left(\sum_{i=1}^s\beta_if_i(n)\right) \ll N^{1-\gamma_s}.\]
Putting everything together we get that 
that
\[\sum_{n\leq N}e\left(\sum_{i=1}^kmp_i\xi_i \lf f_i(n)\rf\right)\ll
  N^{1-\gamma_s/2}.\]
Thus contradicting the lower bound for $\eta>\gamma_s/2$.

Now we return to the case that there is no $1\leq s\leq k$ such that
$\lVert m\xi_s\rVert\geq N^{\gamma_s-\deg f_s}$. Similar to the
one-dimensional case we may not apply the indicator
function. Therefore we show that there is a joint multiple of $m$
directly contradicting that there is no element near a point of
$\Lambda$. To this end let
\begin{align*}
  f_1(n)&=a_{1,\ell}m^\ell+\cdots+a_{1,1}m+a_{1,0}+a_{1,-1}m^{-1}+\cdots\\
  \vdots& \quad\vdots\\
  f_s(n)&=a_{s,\ell}m^\ell+\cdots+a_{s,1}m+a_{s,0}+a_{s,-1}m^{-1}+\cdots
\end{align*}   
Then $m\mid \lf f_i(n)\rf$ for $i=1,\ldots,k$ if and only if
$a_{i,0}=0$ for $i=1,\ldots,k$ if and only if
$m^{-1}f_i(n)\in\left[0,\frac1m\right)$ for $i=1,\ldots,k$.

Let $\mathcal{N}(\mathbf{f},X)$ denote the number of $1\leq n\leq X$
such that $\lvert f_i(n)\rvert$ is a multiple of $m$ for
$1\leq i\leq k$. As above $X$ will be a power of $M$ and thus of
$N$. If $\chi$ denotes the indicator function of the interval
$\left[0,\frac1m\right)$, then an application of Vaaler
polynomials~\ref{vaaler:indicator:function} yields
\begin{align*}
  &\lvert \mathcal{N}(\mathbf{f},X)-\frac{X}{m^k}\rvert\\
  &\leq \sum_{1\leq n\leq
    X}\lvert\prod_{i=1}^k\chi\left(\frac{f_i(n)}{m}\right)
      -\prod_{i=1}^k\chi^*\left(\frac{f_i(n)}{m}\right)\rvert
    +\sum_{1\leq n\leq
    X}\lvert\prod_{i=1}^k\chi^*\left(\frac{f_i(n)}{m}\right)-\frac{1}{m^k}\rvert\\
  &\leq N\left(\left(1+\frac1{H+1}\right)^k-1\right)
    +  \sum_{0<\lVert\mathbf{h}\rVert_\infty\leq H}\frac1{r(\pi\mathbf{h})}\lvert\sum_{n\leq X}
    e\left(\frac1m\sum_{i}h_if_i(n)\right)\rvert\\
  &\quad  +\sum_{0<\lVert\mathbf{h}\rVert_\infty\leq
      H}\left(\frac1{H+1}\right)^{k-\beta(\mathbf{h})}
      \left(1+\frac1{H+1}\right)^{\beta(\mathbf{h})}\lvert\sum_{n\leq
    X}e\left(\frac1m\sum_{i}h_i f_i(n)\right)\rvert.
\end{align*}
Setting $X=M^{\frac1{1-\rho}}$ and $H=X^{\frac1{10}}$ we get by the
same reasoning as in the one-dimensional case that for sufficiently
large $X$ we have $\mathcal{N}(\mathbf{f},X)\geq1$.  Thus similar to
the one-dimensional case there is a $1\leq n\leq N$ such that
$\lVert \xi_i\lf f_i(n)\rf\rVert$ is very close to an
integer. Multiplying by the primitive vector $\mathbf{p}$ we get that
this is very close to a point in the lattice $\Lambda$ violating the
assumption.
\end{proof}

\begin{proof}
  The proof of Theorem \ref{thm:multiple_prime_heilbronn_set} follows
  along the same lines as the proof of
  Theorem~\ref{thm:multiple_lattice_approximation}. The only change is
  in the exponentials sums which run over the primes and the
  corresponding estimates.
\end{proof}

\section*{Final remarks}
In the present paper we have considered one of the many examples of
van der Corput sets provided in Section \ref{sec:van-der-corput}. Each
of these examples lead to different exponential sums whose treatment
is interesting on their own. In the vain of Bergelson \textit{et al.}
\cite{bergelson_kolesnik_madritsch+2014:uniform_distribution_prime},
where mixtures of polynomials and pseudo-polynomials were considered,
similar results should hold for Heilbronn sets. For example, let $f$
be a polynomial with real coefficients. Then we suppose that one can
prove the existence of an $\eta>0$ such that
\[\min_{1\leq n\leq N}\lVert\xi\lfloor f(n)\rfloor\rVert\ll
N^{-\eta}\]
for any given $\xi\in\R$ and any given $N\in\N$. Similar statements
should be true for the sequence $\lfloor f(p)\rfloor$ and
multi-dimensional variates thereof.

\section*{Acknowledgment}

Parts of this research work were done when the first author was a
visiting professor at the Department of Analysis and Number Theory,
Graz University of Technology. The author thanks this institution and
the doctoral school ``Discrete Mathematics'' funded by
FWF. Furthermore parts of the present paper were worked out when the
authors were at the conference ``Normal numbers: Arithmetic,
Computational and Probabilistic Aspects'' at the Erwin Schrödinger
Institute in Vienna. They thank the institute for its hospitality
and facilities.

For the realization of the present paper the first author received
support from the Conseil R\'egional de Lorraine. The second author
acknowledges support of the project F 5510-N26 within the special
research area ``Quasi Monte-Carlo Methods and applications'' funded
by FWF (Austrian Science Fund).

\section*{Acknowledgment to the referee}

We are grateful to an anonymous referee for his careful reading of the
first version of the present paper.



\end{document}